%% file: lamy_zimmermann_v3_Nov_29.tex
\documentclass[reqno,oneside,11pt]{amsart}

\usepackage[a4paper, top=2.5cm, bottom=2.5cm, left = 3.25cm, right = 3.25cm]{geometry}

\usepackage[utf8]{inputenc}
\usepackage[T1]{fontenc}
\usepackage[svgnames,hyperref]{xcolor}
\usepackage{lmodern}
\usepackage{amssymb} 
\usepackage{graphicx}
\usepackage{import} 
\usepackage{tikz}

\DeclareFontFamily{U}{mathb}{\hyphenchar\font45}
\DeclareFontShape{U}{mathb}{m}{n}{
      <5> <6> <7> <8> <9> <10> gen * mathb
      <10.95> mathb10 <12> <14.4> <17.28> <20.74> <24.88> mathb12
      }{}
\DeclareSymbolFont{mathb}{U}{mathb}{m}{n}
\DeclareMathSymbol{\bigast}{1}{mathb}{"06}

\DeclareFontFamily{U}{mathx}{\hyphenchar\font45}
\DeclareFontShape{U}{mathx}{m}{n}{<-> mathx10}{}
\DeclareSymbolFont{mathx}{U}{mathx}{m}{n}
\DeclareMathAccent{\widebar}{0}{mathx}{"73}

\usepackage[cmtip]{xy}
\xyoption{pdf}
\xyoption{color}
\xyoption{all}
\newcommand{\mygraph}[1]{\xybox{\xygraph{#1}}}

\usepackage[shortlabels]{enumitem}
\setlist[enumerate]{label=\rm{(\arabic*)}}
\setlist[enumerate,2]{label=\rm({\it\roman*})}
\setlist[itemize]{label=\raisebox{0.25ex}{\tiny$\rond$}}

\theoremstyle{plain}
\newtheorem{maintheorem}{Theorem}

\newtheorem{maincorollary}[maintheorem]{Corollary}
\newtheorem{theorem}{Theorem}[section]
\newtheorem{corollary}[theorem]{Corollary}
\newtheorem{proposition}[theorem]{Proposition}
\newtheorem{lemma}[theorem]{Lemma}

\theoremstyle{definition}

\newtheorem{example}[theorem]{Example}
\newtheorem{remark}[theorem]{Remark}
\newtheorem{setup}[theorem]{Set-Up}

\newcommand{\A}{\mathbb{A}}
\newcommand{\F}{\mathbb{F}}
\newcommand{\p}{\mathbb{P}}

\newcommand{\Z}{\mathbf{Z}}

\newcommand{\K}{\mathbf{k}}

\newcommand{\C}{\mathbf{C}}
\newcommand{\R}{\mathbf{R}}
\newcommand{\Q}{\mathbf{Q}}
\newcommand{\FF}[1]{\mathbf{F}_{\!#1}} 

\newcommand{\E}{\mathcal E}
\newcommand{\B}{\mathcal B}

\newcommand{\Al}{\mathcal A} 
\newcommand{\Bl}{\mathcal B} 
\newcommand{\Cl}{\mathcal C} 
\newcommand{\Fl}{\mathcal F} 
\newcommand{\Gl}{\mathcal G} 
\newcommand{\Il}{\mathcal I} 
\newcommand{\Tl}{\mathcal T} 
\newcommand{\Wl}{\mathcal W} 
\newcommand{\Xl}{\mathcal X} 
\newcommand{\Yl}{\mathcal Y} 

\newcommand{\Ol}{\mathcal O}

\DeclareMathOperator{\PGL}{PGL}
\DeclareMathOperator{\PSL}{PSL}

\newcommand{\eps}{{\varepsilon}}
\renewcommand{\phi}{\varphi}
\newcommand{\id}{\text{\rm id}}
\newcommand{\rat}{\dashrightarrow}
\def\dashmapsto{\mapstochar\dashrightarrow}

\def\mapsrat{\mapstochar\dashrightarrow}

\newcommand{\llangle}{\mathopen\ll}
\newcommand{\rrangle}{\mathclose\gg}
\renewcommand{\setminus}{\smallsetminus}
\newcommand{\rond}{\scalebox{1.2}{$\bullet$}}

\DeclareMathOperator{\Aut}{Aut}
\DeclareMathOperator{\Bir}{Bir}
\DeclareMathOperator{\Gal}{Gal}
\DeclareMathOperator{\Stab}{Stab}

\DeclareMathOperator{\Spec}{Spec}

\DeclareMathOperator{\Div}{Div}

\DeclareMathOperator{\Proj}{Proj}

\DeclareMathOperator{\Conv}{Conv}

\DeclareMathOperator{\nef}{nef}
\DeclareMathOperator{\Eff}{Eff}
\DeclareMathOperator{\Int}{Int}
\DeclareMathOperator{\Star}{star}
\DeclareMathOperator{\order}{order}
\DeclareMathOperator{\NE}{NE}

\newcommand{\ConeConv}{\Conv^{\circ}}
\newcommand{\pt}{\text{pt}}

\renewcommand{\le}{\leqslant}
\renewcommand{\ge}{\geqslant}

\newcommand{\mycolor}{Navy}
\usepackage[pdfauthor={S. Lamy and S. Zimmermann}, colorlinks, citecolor = \mycolor, linkcolor = \mycolor, urlcolor = \mycolor]{hyperref}
\usepackage[all]{hypcap} 
\renewcommand{\L}{\mathbf L}

\newcommand\SB[1][\scalebox{.8}]{#1} 
\newcommand\SBd[1][\scalebox{.9}]{#1} 

\title[Signature morphisms from the Cremona group]{Signature morphisms from the Cremona group over a non-closed field}
\author{Stéphane Lamy}
\author{Susanna Zimmermann}
\address{Institut de Math\'ematiques de Toulouse, Universit\'e Paul Sabatier,
118 route de Narbonne, 31062 Toulouse Cedex 9, France}
\email{slamy@math.univ-toulouse.fr}
\email{susanna.zimmermann@univ-angers.fr}
\curraddr{LAREMA, Universit\'e d'Angers, 49045 Angers cedex 1, France}
\date{\today}

\begin{document}

\thanks{The second author greatefully acknowledges support by the Swiss National Science Foundation.}

\begin{abstract}
We prove that the plane Cremona group over a perfect field with at least one Galois extension of degree $8$ is a non-trivial amalgam, and that it admits a surjective morphism to a free product of groups of order two.
\end{abstract}

\maketitle


\section*{Introduction}

The Cremona group $\Bir_\K(\p^2)$ is the group of birational symmetries of the projective plane defined over a field $\K$. 
Its elements are of the form
\[[x:y:z]\dashmapsto[f_0(x,y,z):f_1(x,y,z):f_2(x,y,z)]\]
where $f_0,f_1,f_2\in\K[x,y,z]$ are homogeneous polynomials of equal degree with no common factor, and such that there exists an inverse of the same form.
Equivalently, working in an affine chart one can define the Cremona group as the group of birational selfmaps of the affine plane, which is also (anti-)isomorphic to the group $\Aut_\K \K(x,y)$ of $\K$-automorphisms of the fraction field $\K(x,y)$.
The Cremona group contains the group of polynomial automorphisms of the affine plane over the field $\K$.
In particular it is a rather huge group.
It is neither finitely generated (see \cite[Proposition 3.6]{CantatSurvey}), nor finite dimensional, even when working over a finite base field.
It was recently shown that $\Bir_{\K}(\p^2)$ is not a simple group, over any base field $\K$ \cite{CL, Lonjou}. 
Then it is natural to ask for nice quotients of $\Bir_{\K}(\p^2)$, for instance abelian ones.
Over an algebraically closed field, it is known (see also \S\ref{subsec:jonq}) that the automorphism group $\Aut_\K(\p^2) = \PGL_3(\K)$, or the Jonquières group $\PGL_2(\K(T)) \rtimes \PGL_2(\K)$, both embed in any quotient of the Cremona group.
In particular in this  situation any morphism from $\Bir_{\K}(\p^2)$ to an abelian group, or to a finite group,  is trivial.
On the other hand it was shown by the second author in \cite{Zimmermann} that the situation is drastically different over the field $\R$ of real numbers. 
The real Cremona group admits an uncountable collection of morphisms to $\Z/2\Z$, and precisely we have the following result about the abelianization of $\Bir_\R(\p^2)$:
\begin{equation} \label{eq:morphism R}
\Bir_\R(\p^2)/[\Bir_\R(\p^2),\Bir_\R(\p^2)] \simeq \bigoplus_{(0,1]} \Z/2\Z. \tag{$\dagger$}
\end{equation}
One consequence is that $\Bir_\R(\p^2)$ is a nontrivial amalgamated
product of two factors along their intersection \cite{Zimmermannb}.

In this paper we explore a similar question, over any perfect base field $\K$ that admits at least one Galois extension of degree 8.
Observe that this condition corresponds to a large collection of fields, which includes the case of all number fields and finite fields.

The special role of degree 8 extensions is explained by their relation to Bertini involutions.
Indeed, given a point of degree 8 on $\p^2$, that is, an orbit of cardinal 8 under the natural action of the absolute Galois group of the base field $\K$, we can consider the surface $S$ obtained by blowing-up this orbit.
If the point is sufficiently general, the surface $S$ is del Pezzo and admits another birational morphism to $\p^2$, and the induced birational selfmap of $\p^2$ is an example of a Bertini involution.
Let $\B\subset\Bir_{\K}(\p^2)$ be a set of representatives of such Bertini involutions with a base point of degree $8$, up to conjugacy by automorphisms.
We prove that as soon as $\K$ admits at least one Galois extension of degree 8, then the set $\B$ is quite large.
Namely $\B$ has at least the same cardinality than $\K$, and in the case of a finite field $\FF{q}$ one can be more precise and give a lower bound for the cardinal of $\B$ which is polynomial in $q$ (see \S\ref{sec:bertini}).

These Bertini involutions are part of a system of elementary generators $\E$ for the group $\Bir_{\K}(\p^2)$ which was found by Iskovskikh \cite{Isko}.
The set $\E$ is always a huge set, because for instance it contains all Jonquières maps (up to left-right composition by automorphisms).
Before stating our results we introduce a bit more notation.
For each Bertini involution $b\in \B$, we set $G_b = \langle\Aut_\K(\p^2), b\rangle$.
Moreover, we denote $G_e = \langle\Aut_\K(\p^2), \E \setminus \B \rangle$ the subgroup generated by  automorphisms and all non-Bertini  elementary generators.
Our first result gives an amalgamated product structure for $\Bir_\K(\p^2)$, in terms of these subgroups.

\begin{maintheorem} \label{thm:main big amalgam}
Let $\K$ be a perfect field admitting at least one Galois extension of degree $8$.
Consider the subgroups $G_i$ as defined above, for $i \in \B\cup\{e\}$. 
Then for all $i\neq j$ we have $G_i\cap G_j = \Aut_\K(\p^2)$, and the Cremona group is the amalgamated product of the $G_i$ along their common intersection:
\[\Bir_\K(\p^2)\simeq\bigast_{\Aut_\K(\p^2)}G_i.\] 
Moreover, $\Bir_\K(\p^2)$ acts faithfully on the corresponding Bass-Serre tree.
\end{maintheorem}

It was shown by Cornulier (appendix of \cite{CL}) that the Cremona group over an algebraically closed field is not a non-trivial amalgam of two groups.
In contrast, we deduce from the above theorem the following structure result, where we denote $G_\B=\langle \Aut_\K(\p^2),\B \rangle$ the subgroup generated by all $G_b$:

\begin{maincorollary} \label{cor:main 2 amalgam}
Let $\K$ be a perfect field admitting at least one Galois extension of degree $8$. Then $G_{e}\cap G_\B = \Aut_\K(\p^2)$, and 
\[\Bir_\K(\p^2)\simeq G_\B \ast_{\Aut_\K(\p^2)}G_{e},\]
and it acts faithfully on its Bass-Serre tree.
\end{maincorollary}

It turns out that each subgroup $G_b$ admits a structure of free product, and this allows to obtain a lot of morphisms from the Cremona group to $\Z/2\Z$:

\begin{maintheorem} \label{thm:main morphisms}
Let $\K$ be a perfect field with at least one Galois extension of degree $8$. 
Then:
\begin{enumerate}
\item For each $b\in \B$ we have $G_b\simeq\Aut_\K(\p^2)\ast\Z/2\Z$, and we can write the Cremona group as a free product
\[
\Bir_\K(\p^2) \simeq G_e \ast \left( \bigast_{\B} \Z/2\Z \right).
\]
\item In particular, there is a surjective morphism 
\[\Bir_\K(\p^2)\rightarrow\bigast_{\B}\Z/2\Z\]
whose kernel is the smallest normal subgroup containing $G_e$, and which sends each $b \in \B$ to the corresponding generator on the right-hand side. 
\item In particular, the abelianization of the Cremona group over $\K$ contains
a subgroup isomorphic to $\bigoplus_{\B}  \Z/2\Z$.
\end{enumerate}
\end{maintheorem}

We see that even if there was no Bertini involutions involved in the paper \cite{Zimmermann}, we obtain a similar looking (even if less precise) result.
In particular, the $\Z/2\Z$ in the target group in (\ref{eq:morphism R}) have nothing to do with the fact that the absolute Galois group of $\R$ has order 2, but rather with the fact that we are able to produce a natural set of generators for the Cremona group that contains involutions.
In this sense, we like to think of the above morphisms $\Bir_\K(\p^2) \to \Z/2\Z$ as some analogues of the classical signature morphism on the symmetric group.
The huge collection of such morphisms corresponds to the existence of a system of generators with a lot of non-conjugate involutions.
In this paper, we focus on Bertini involutions associated to a base point of degree 8 because they seem to be the easiest to handle technically. 
When the base field is $\R$, a similar role was played by the so-called ``standard quintic involutions''. 
It seems quite plausible that other ``signature morphisms'' exist on the Cremona group, associated to other type of involutions, such as the Geiser involution associated to a base point of degree 7.
Also, we mention that we see no obvious obstruction why such morphisms could not exist in higher dimension, even over the field of complex numbers.

The strategy to prove the above results is to use the Sarkisov Program.
The Sarkisov Program is a way to factorize a given birational map between Mori fiber spaces into elementary links.
We recall that even if the starting map is a birational selfmap of a given variety $X$ (for instance $X = \p^2$), the elementary links are not in general elements of the group $\Bir(X)$.
In other words, even if one is primarily interested in the group $\Bir(X)$, the Sarkisov Program naturally produces generators for the groupoid of birational maps $Y \rat Z$, where $Y, Z$ can be any Mori fiber spaces birational to $X$.
Nevertheless the Sarkisov Program turns out to be an efficient tool to produce some systems of generators for $\Bir_\K(\p^2)$, and also to describe relations between them \cite{Isko, IKT, iskovskikh_1996}.
The Sarkisov Program was revisited recently in light of the progresses in the theory of the Minimal Model Program, and is now established in any dimension (over $\C$) \cite{HMcK}.
Moreover the relations between Sarkisov links were described by Kaloghiros \cite{Kaloghiros}.

In Section \ref{sec:squares} we encode Sarkisov links and relations between them in a square complex $\Xl$ on which the group $\Bir_\K(\p^2)$ acts naturally.
Then in Section \ref{sec:sarkisov} we give an account of the proof of the Sarkisov Program in the simpler case of surfaces, but working over an arbitrary perfect field.
This allows to prove that the square complex $\Xl$ is connected and simply connected.

In Section \ref{sec:generators} we recall the notion of elementary generators for the group $\Bir_\K(\p^2)$, following the work of Iskovskikh.
Among the elementary generators we discuss in particular the Bertini involutions and prove their existence (and in fact, their abundance).
We also discuss the Jonquières maps, and in any dimension we recall the following basic dichotomy: given a morphism $\phi$ from the Cremona group to another group $H$, either the subgroup generated by the Jonquières maps lies in the kernel, or $\phi$ induces an embedding of this subgroup into $H$.

Finally in Section \ref{sec:proofs} we use Bass-Serre theory to prove our results. 
The general idea is that the Bass-Serre trees of the various amalgams appearing in Theorem \ref{thm:main big amalgam}, Corollary \ref{cor:main 2 amalgam} and Theorem \ref{thm:main morphisms} are realized either as a quotient or as a subcomplex of the square complex $\Xl$.

When one encounters a cube complex in geometric group theory, a natural question is whether this complex has non-positive curvature.
It turns out that this is not the case for our square complex $\Xl$, however we should mention that $\Xl$ is essentially a subcomplex of an infinite dimensional CAT$(0)$ cube complex associated with the Cremona group that was constructed by Lonjou in her PhD thesis.
We thank Anne Lonjou for many useful discussions at an early stage of this project, Anne-Sophie Kaloghiros for clarifying to us some fine points in her work \cite{Kaloghiros}, and J\'er\'emy Blanc and Andrea Fanelli for discussions on issues with birational maps over non-perfect fields.

\section{Birational maps between surfaces over an arbitrary field}
\label{sec:surfaces}

In this section we review some results about the birational geometry of
surfaces, with a focus on the case of an arbitrary perfect base field.

\subsection{Factorization into blow-ups}\label{ssec:1.1}

Let $\K$ be a perfect field, and $\K^a$ an algebraic closure.
All field extensions of $\K$ that we shall consider will be supposed to lie in $\K^a$.
By a \textit{surface} (over $\K$) we shall mean a smooth projective surface defined over $\K$.
We denote by $S(\K)$ the set of $\K$-rational points on $S$.
The Galois group $\Gal(\K^a/\K)$ acts on $S \times_{\Spec \K} \Spec \K^a$ through the second factor.
In particular, $\Gal(\K^a/\K)$ acts on the set $S(\K^a)$ of $\K^a$-rational points.
By a \textit{point of degree $d$} on $S$ we mean an orbit $p =\{p_1, \dots, p_d\}\subset S(\K^a)$ of cardinal $d$ under the action of $\Gal(\K^a/\K)$.
Observe that the points in $S(\K)$ are exactly the fixed points for the action of $\Gal(\K^a/\K)$ on $S(\K^a)$, or in other words the
points of degree 1.
Let $\L/\K$ be a field extension such that the $p_i$ are $\L$-rational points.
We call blow-up of $p$ the blow-up of these $d$ points, which is a morphism $\pi\colon S' \to S$ defined over $\K$, with exceptional divisor $E = C_1 + \dots + C_d$, where the $C_i$ are disjoint $(-1)$-curves defined over $\L$, and $E^2 = -d$. 
We shall refer to this situation by saying that $E$ is an \textit{exceptional divisor of degree $d$}.

We recall the following classical factorization results (see e.g. \cite[Theorems 9.2.2 and 9.2.7]{Liu}). 

\begin{proposition} \label{pro:factorization}
Let $\pi\colon S' \to S$ be a birational morphism between surfaces defined over $\K$.
Then $\pi = \pi_1 \circ \dots \circ \pi_n$, where each $\pi_i \colon S_i \to S_{i-1}$ is the blow-up of a point of degree $d_i \ge 1$ on $S_{i-1}$, with exceptional divisor $E_i$ on $S_i$ satisfying $E_i^2 = -d_i$ $($in particular $S = S_0$, and $S' = S_n)$.
\end{proposition}

\begin{proposition} \label{pro:resolution}
Let $\phi\colon S \rat S'$ be a birational map between surfaces defined over  $\K$.
Then there exist $Z$ a surface defined over $\K$, and sequences of blow-ups $\pi\colon Z \to S$, $\pi'\colon Z \to S'$ of orbits of points
under $\Gal(\K^a/\K)$, such that $\pi' = \phi \circ \pi$.
\end{proposition}

We should mention that even if the Cremona group was explicitly defined in the introduction in terms of homogeneous polynomials, in practice we almost always think of an element of $\Bir_\K(\p^2)$ as given by two sequences of blow-ups defined over $\K$, as provided by Proposition \ref{pro:resolution}.

\begin{remark}
Over a non-perfect field $\K$, there is no reason why the base points of a birational map should be defined over a separable closure of $\K$, and so we can no longer identify closed points with Galois orbits as we did in the statements of Propositions \ref{pro:factorization} and \ref{pro:resolution}.
As a simple example of this phenomenon, consider $\K = \FF{2}(t)$, and denote by $t^{\frac{1}{2}}$ the unique square root of $t$ in $\K^a$. 
Then $\K(t^{\frac{1}{2}})/\K$ is a non separable extension.
Now consider the birational involution $f \in \Bir_\K(\p^2)$ given by 
\[
f\colon [x_0:x_1:x_2]\dashmapsto[x_0x_2:x_1x_2:x_0^2+tx_1^2]\]
As a quadratic birational map, $f$ admits 3 base points defined over $\K^a$, which are
\begin{align*}
p_1 = [0:0:1], && p_2=  [t^{\frac{1}{2}}:1:0],
\end{align*}
and $p_3$ a point infinitely near to $p_2$.
In particular $p_2$ is not defined over a separable extension of $\K$.
\end{remark}

\subsection{Negative maps, minimal and ample models, scaling}

Let $S$ be a surface defined over $\K$, and $S^a$ the same surface over $\K^a$.
We define the Néron-Severi space $N^1(S^a)$ as the space of numerical classes of $\R$-divisors:
\[N^1(S^a) := \Div(S^a) \otimes \R/\equiv.\]
The action of $\Gal(\K^a/\K)$ on $N^1(S^a)$ factors through a finite group, and we denote by $N^1(S)$ the subspace of invariant classes.
Since we only consider surfaces with $S(\K) \neq \emptyset$ and $\K^a[S^a]^* = (\K^a)^*$, $N^1(S)$ is also the space of classes of divisors defined over $\K$ (see \cite[Lemma 6.3(iii)]{Sansuc}).
The dimension of this finite dimensional $\R$-vector space is called the Picard number of $S$ over $\K$, and denoted by $\rho(S)$. 

\begin{remark}
When working on a surface $S$, we can identify the space $N^1(S)$ of divisors and the space $N_1(S)$ of 1-cycles, and similarly the subspaces $\Eff(S)$ or $\NE(S)$ of effective divisors or 1-cycles.
In the sequel we shall use the notation that seems most natural in view of the extension of the results in higher dimension.
For instance the Cone Theorem \ref{thm:cone} is about 1-cycles, so there we use the notation $\NE(S)$.
\end{remark}

Let $\pi\colon S' \to S$ be a birational morphism between surfaces defined over $\K$, and $D'$ a $\Q$-divisor on $S'$ with push-forward $D = \pi_*(D')$.
By Proposition \ref{pro:factorization}, we can write $\pi = \pi_1 \circ \dots \pi_n$, where $\pi_i\colon S_i \to S_{i-1}$ is the blow-up of a point of degree $d_i$, with $S = S_0$ and $S' = S_n$.
For any $i$, we denote by $E_i$ the exceptional divisor of $\pi_i$, and by $D_i$ the push-forward of $D'$ on $S_i$.
We say that $\pi$ is \textit{$D'$-negative} if $D_i \cdot E_i < 0$ for all $i$.
Observe also that on $S'$ we can write 
\[D' = \pi^* D + \sum a_i E_i\]
for some $a_i \in \Q$, where here the $E_i$ denote strict transforms on $S'$.

\begin{lemma} \label{lem:negative}
With the above notation, the morphism $\pi$ is $D'$-negative if and only if $a_i > 0$ for all $i$.
\end{lemma}

\begin{proof}
On $S_i$, we have $D_i = \pi_i^* D_{i-1} + a_i E_i$, so that 
\[0 = \pi^* D_{i-1} \cdot E_i = D_i \cdot E_i - a_i E_i^2.\]
Since $E_i^2 = -d_i$, where $d_i \ge 1$ is the degree of the point blown-up by $\pi_i$, we get $a_i = - \frac{D_i \cdot E_i}{d_i}$, so that $a_i$ and $D_i \cdot E_i$ have opposite signs as expected.
\end{proof}

If $\pi\colon S' \to S$ is a $D'$-negative birational morphism, and $D = \pi_*(D')$ is nef, we call $S$ a \textit{$D'$-minimal model} of $S'$.
Such a model, if it exists, is unique:

\begin{lemma}[{\cite[p. 94]{Matsuki}}] \label{lem:unique model}
Let $S_1$, $S_2$ be two $D'$-minimal models of $S'$, then the induced map $S_1 \rat S_2$ is an isomorphism.
\end{lemma}

We shall use the above setting for divisors $D'$ of the form $D' = K_{S'} + A$, with $A$ an ample $\Q$-divisor (or $A=0$).
Observe that a $(K_{S'} + A)$-negative birational morphism is also a $K_{S'}$-negative morphism, hence simply a sequence of inverse of blow-ups. 
There exist surfaces with infinitely many $(-1)$-curves: a classical example is given by $\p^2$ blown-up at (sufficiently general) 9 points. 
This gives a countable collection of curves $C_i$ with $K_{S'} \cdot C_i <0$ and $C_i^2 <0$.
However, after perturbing $K_{S'}$ by adding any $\Q$-ample divisor, we get a finite collection:

\begin{theorem}[{\cite[Cone Theorem D.3.2]{Reid}}] \label{thm:cone}
Let $S$ be a surface defined over $\K^a$.
Then if $\rho(S) \ge 3$, all $K_S$-negative extremal rays of the cone $\NE(S)$ are of the form $\R_{>0} C$ with $C$ a $(-1)$-curve.
Moreover, for any ample $\Q$-divisor $A$ on $S$, there are only finitely many $(-1)$-curves $C_i$ such that $(K_S + A)\cdot C_i < 0$.
\end{theorem}

\begin{proof}[Comments on the proof]
In the Cone Theorem, the main delicate point to check when working in arbitrary characteristic is vanishing.
If $A$ is ample, by Serre duality we have $H^2(S,K_S+A) = H^0(S,-A)^* = 0$.
So by Riemann-Roch on a surface we have, for the divisor $D = K+A$:
\begin{equation} \label{RRinequality}
h^0(S,D) \ge h^0(S,D) - h^1(S,D) = \frac12 D(D-K_S) +
\chi(\Ol_S).
\end{equation}
In fact when $S$ is rational, we even have  $H^1(S,K_S+A) = 0$ (Kodaira vanishing in positive characteristic can only fail for surfaces of Kodaira dimension $\ge 1$, see \cite[Theorem 1.6]{Terakawa}).
But this extra information is not necessary in the argument given by Reid, as Inequality (\ref{RRinequality}) is enough.
\end{proof}

\begin{remark} \phantomsection \label{rem:2 rays}
\begin{enumerate}[wide]
\item Over an arbitrary perfect field $\K$, we have an equivariant version of the Cone Theorem with respect to the action of $\Gal(\K^a/\K)$, see \cite[p. 48]{KM}.
Essentially we only have to change ``$(-1)$-curve'' by ``orbit of pairwise disjoint $(-1)$-curves under the action of $\Gal(\K^a/\K)$''.
By Castelnuovo Contraction Theorem, we can contract such an orbit and obtain a new smooth projective surface.
Thus by running the Minimal Model Program with respect to the canonical divisor $K$, or more generally with respect to $K+A$ with $A$ ample, we stay in the category of smooth surfaces, and at posteriori this justifies that we restrict ourselves to this setting.

\item In the case where the Picard number $\rho(S)$ is equal to 2, $\NE(S)$ is a convex cone in a real 2-dimensional vector space, thus we have at most two extremal rays, which correspond either to the contraction of an exceptional divisor or to a Mori fibration.
This case is particularly interesting in a relative setting, and is then often referred to as a \textit{two rays game}.
Precisely, we start with a morphism $\pi\colon S \to Y$ from a surface $S$, with relative Picard number equal to 2. 
We assume that any curve $C$ contracted by $\pi$ satisfies $K_S \cdot C < 0$.
Then there exist exactly two morphisms of relative Picard number 1, $\pi_i \colon S \to Y_i$, $i=1,2$,  such that $\pi$ factors through each of the $\pi_i$:
\[\xymatrix@R-15pt@C-15pt{
& S \ar[dl]_{\pi_1} \ar[dr]^{\pi_2}  \ar[dd]^{\pi} \\
Y_1 \ar[dr] && Y_2 \ar[dl] \\
& Y
}\]
\end{enumerate}

\end{remark}

\begin{theorem}[Base Point Free Theorem, see {\cite[D.4.1]{Reid}}]
\label{thm:BPF}
Let $S$ be a surface, and let $D$ be a nef divisor on $S$ such that $D - \eps K_S$ is ample for some $\eps > 0$. 
Then the linear system $|mD|$ is base point free for all sufficiently large $m$.
\end{theorem}

If $D$ is nef and of the form $D = K + A$ with $A$ ample, we denote by $\phi_D$ the morphism from $S$ associated to the linear system $|mD|$ for $m \gg 0$.
This morphism has connected fibers, and it contracts precisely the curves $C$ such that $(K+A)\cdot C = 0$.
Now we extend the definition of $\phi_D$ to any pseudo-effective divisor $D = K + A$, by using the notion of scaling.

Let $S$ be a surface, and $A$, $\Delta$ ample $\Q$-divisors on $S$ (we also admit the case $A = 0$).
The $(K+A)$ \textit{minimal model program with scaling of} $\Delta$ is a sequence of birational morphisms $\pi_i\colon S_{i-1} \to S_{i}$ defined iteratively as follows.
We set $S_0 = S$.
If $S_i$ is constructed, we denote $A_i$ and $\Delta_i$ the direct images of $A$ and $\Delta$ on $S_i$, and we consider $t_i$ such that $D_i = K_{S_{i}} + A_{i} + t_i \Delta_{i}$ is nef but not ample on $S_i$.
Then the morphism $\pi_{i+1}$ from $S_i$ is obtained by applying Theorem \ref{thm:BPF} to $D_i$.
If $\pi_{i+1}(S_i)  = S_{i+1}$ is a surface, we repeat the construction.
At some point $\pi_{i+1}$ is a fibration to a curve or a point, in which case we reached a Mori fiber space and the program stops.
In particular this process gives a finite sequence of rational numbers 
\[t_0 > t_1 > \dots > t_n\]
such that $K_{S_i} + A_i + t \Delta_i$ is ample on $S_i$ for any $t > t_i$.
We shall say that  a birational morphism $\pi\colon S \to S'$ is $(K+A)$-negative with scaling of $\Delta$ if $S'$ is one of the $S_i$ in the above process.

Now assume that $D$ is a pseudo-effective $\Q$-divisor on $S$ of the form $D = K + \Delta$, with $\Delta$ ample.
We run the $K$ minimal model program with scaling of $\Delta$, and we look where the coefficient $t = 1$ corresponding to $D$ fits into the sequence $t_0 > t_1 > \dots > t_n$. 
Precisely, we set $j = \max\{i \mid t_i \ge 1 \}$ and we denote $\phi_D = \pi_j \circ \dots \circ \pi_1$.
Observe that the morphism $\phi_D$ is birational if and only if $j \le n-1$, and in any case $\phi_D(D)$ is ample.
We say that $\phi_D$ is the \textit{ample model} of $D$.

\begin{remark} \phantomsection \label{rem:phi D}
\begin{enumerate}[wide]
\item 
The above construction only depends on the numerical class of $D$, which would not be the case for more general $D$ (that is, not of the form $K + \Delta$ with $\Delta$ ample), see \cite[Example 4.8]{KKL}.
\item 
The morphism $\phi_D$ coincides with the morphism from $S$ to $\Proj(\bigoplus H^0(Z, mD))$, see \cite[Remark 2.4]{KKL}.
In particular, if we write $D = K + \Delta_1 + \Delta_2$ with $ \Delta_1, \Delta_2$ ample, and run the $K+\Delta_1$ minimal model program with scaling of $\Delta_2$, we will get the same morphism $\phi_D$, but possibly by another sequence of contractions.
\end{enumerate}
\end{remark}

\section{A square complex associated to the Cremona group} \label{sec:squares}

In this section we construct a square complex that encodes Sarkisov links and relations between them. 
First we introduce the key notion of rank $r$ fibration.

\subsection{Rank \textit{r} fibrations}

If not stated otherwise, all varieties and morphisms are defined over $\K$.
Let $S$ be a surface, and $r \ge 1$ an integer.
We say that $S$ is a \textit{rank $r$ fibration} if there exists a surjective morphism $\pi \colon S \to B$ with connected fibers, where $B$ is a point or a smooth curve, with relative Picard number equal to $r$, and such that the anticanonical divisor $-K_S$ is $\pi$-ample.
The last condition means that for any curve $C$ contracted to a point by $\pi$, we have $K_S \cdot C < 0$.
Observe that the condition on the Picard number translates as $\rho(S) = r$ if $B$ is a point, and $\rho(S) = r+1$ if $B$ is isomorphic to $\p^1$.
If $S$ is a rank $r$ fibration, we will write $S/B$ if we want to emphasize the basis of the fibration, and  $S^r$ when we want to emphasize the rank.
An isomorphism between two fibrations $S/B$ and $S'/B'$ (necessarily of the same rank $r$) is an isomorphism $S \stackrel{\sim}{\to} S'$ such that there exists an isomorphism on the bases (necessarily uniquely defined) that makes the following diagram commute: 
\[\xymatrix@R-5pt@C-5pt{
S \ar[r]^{\sim} \ar[d]_{\pi} & S ' \ar[d]^{\pi'} \\
B \ar[r]^{\sim} & B'
}\]

As the following examples make it clear, there are sometimes several choices for a structure of rank $r$ fibration on a given surface, that may even correspond to distinct ranks.

\begin{example}
\begin{enumerate}[wide]
\item 
$\p^2$ with the morphism $\p^2 \to \pt$, or the Hirzebruch
surface $\F_n$ with the morphism $\F_n \to \p^1$, are rank 1 fibrations.
\item 
$\F_1$ with the morphism $\F_1 \to \pt$ is a rank 2 fibration.
Idem for $\F_0 \to \pt$.
The blow-up $S^2 \to \F_n \to \p^1$ of a Hirzebruch surface along a point of degree $d$, such that each point of the orbit is in a distinct fiber, is a rank 2 fibration over $\p^1$.   
\item The blow-up of two distinct points on $\p^2$, or of two points of $\F_n$ not lying on the same fiber, give examples of rank 3 fibrations, with morphisms to the point or to $\p^1$ respectively.
\end{enumerate}
\end{example}

\begin{remark}
Observe that the definition of a rank $r$ fibration puts together several well-known notions.
If $B$ is a point, then $S$ is a del Pezzo surface of Picard rank $r$ (over the base field $\K$).
If $B$ is a curve, then $S$ is a conic bundle of relative Picard rank $r$: a general fiber is isomorphic to $\p^1$, and (over $\K^a$) any singular fiber is the union of two $(-1)$-curves secant at one point.
Remark also that rank 1 fibrations are exactly the usual 2-dimensional Mori fiber spaces.
\end{remark}

We will be interested only in rational surfaces, and we call \textit{marking} on a rank $r$ fibration $S/B$ a choice of a birational map $\phi\colon S \rat \p^2$.
Observe that if $S$ is rational and $B$ is a curve, then $B$ is isomorphic to $\p^1$.
We say that two marked fibrations $\phi \colon S/B \rat \p^2$ and $\phi' \colon S'/B' \rat \p^2$ are \textit{equivalent} if $\phi'^{-1} \circ \phi \colon S/B \to S'/B'$ is an isomorphism of fibrations.
We denote by $(S/B, \phi)$ an equivalence class under this relation.
The Cremona group $\Bir_{\K}(\p^2)$ acts on the set of equivalence classes of marked fibrations by post-composition:
\[f\cdot (S/B,\phi) := (S/B,f\circ \phi).\]

If $S'/B'$ and $S/B$ are marked fibrations of respective rank $r' > r \ge 1$, we say that $S'/B'$ \textit{factorizes} through $S/B$ if the birational map $S' \to S$ induced by the markings is a morphism, and moreover there exists a (uniquely defined) morphism $B \to B'$ such that the following diagram commutes:
\begin{equation} \label{dia:factorize}
\vcenter{
\xymatrix@R-5pt{
S' \ar[rrr]^{\pi'} \ar[dr]  &&&  B' \\
&S \ar[r]^{\pi} & B \ar[ur]\\
}
}
\end{equation}
In fact if $B' = \pt$ the last condition is empty, and if $B' \simeq \p^1$ it means that $S' \to S$ is a morphism of fibration over a common basis $\p^1$.

\subsection{Square complex}\label{subsec:full complex}

We define a 2-dimensional complex $\Xl$ as follows.
Vertices are equivalence classes of marked rank $r$ fibrations, with $3 \ge r\ge 1$. 
We put an oriented edge from $(S'/B',\phi')$ to $(S/B,\phi)$ if $S'/B'$ factorizes through $S/B$.
If $r' > r$ are the respective ranks of $S'/B'$ and $S/B$, we say that the edge has type $r',r$.
For each triplets of pairwise linked vertices $(S''^3/B'',\phi'')$, $(S'^2/B', \phi')$, $(S^1/B, \phi)$, we glue a triangle. 
In this way we obtain a 2-dimensional simplicial complex $\Xl$ on which the Cremona group acts.

\begin{lemma} \label{lem:2 triangles}
For each edge of type $3,1$ from $S''/B''$ to $S/B$, there exist exactly two triangles that admit this edge as a side.  
\end{lemma}

\begin{proof}
In short, the proof is a two rays game (see Remark \ref{rem:2 rays}).
By assumption $S''/B''$ factorizes through $S/B$, so by setting $Y = S$ (if $B \simeq B'$) or $Y = B$ (if $B \simeq \p^1$ and $B' = \pt$), we obtain via Diagram (\ref{dia:factorize}) a morphism $\pi\colon S'' \to Y$ with relative Picard number $\rho(S''/Y)$ equal to 2.
We have exactly two extremal rays in the cone $\NE(S''/Y)$, and since $\rho(S'') = 3$ or 4, both correspond to divisorial contractions.
Denote by $S'' \to S'$ and $S'' \to \tilde S'$ these two contractions. 
Then the two expected triangles are $S''/B''$, $S'/B'$, $S/B$ and $S''/B''$, $\tilde S'/B'$, $S/B$.
\end{proof}

In view of the lemma, by gluing all the pairs of triangles along edges of type 3,1, and keeping only edges of types 3,2 and 2,1, we obtain a square complex that we still denote $\Xl$.
We call \textit{vertices of type $\p^2$} the vertices in the orbit of the vertex $(\p^2/\pt, \id)$ under the action of $\Bir_\K(\p^2)$.
When drawing subcomplexes of $\Xl$ we will often drop part of the information which is clear by context, about the markings, the equivalence classes and/or the fibration.
For instance $S/B$ must be understood as $(S/B, \phi)$ for an implicit marking $\phi$, and $(\p^2, \phi)$ as $(\p^2/\pt, \phi)$.

\begin{example} \label{exple:complex 1}
Let $S$ be the surface obtained by blowing-up $\p^2$ in two distinct points $a$ and $b$ of degree 1.
Denote by $\F_{1,a}/\p^1_a$, $\F_{1,b}/\p^1_b$ the two intermediate Hirzebruch surfaces with their fibrations to $\p^1$.
Finally, denote by $\F_0$ the surface obtained by contracting the strict transform on $S$ of the line through $a$ and $b$.
All these surfaces fit into the subcomplex of $\Xl$ pictured on Figure \ref{fig:exple1}, where the dotted arrows are the edges of type 3,1 that we need to remove from the simplicial complex in order to get a square complex.

\begin{figure}[ht]
\[
\mygraph{
!{<0cm,0cm>;<1.5cm,0cm>:<0cm,1.8cm>::}
!{(0,0)}*+{S/\pt}="T"
!{(-1,-.5)}*+{S/\p^1_a}="Sa"
!{(1,-.5)}*+{S/\p^1_b}="Sb"
!{(-1,-1.5)}*+{\F_0/\p^1_a}="F0a"
!{(1,-1.5)}*+{\F_0/\p^1_b}="F0b"
!{(0,-1)}*+{\F_0/\pt}="F0p"
!{(-1,.5)}*+{\F_{1,a}/\pt}="F1ap"
!{(1,.5)}*+{\F_{1,b}/\pt}="F1bp"
!{(-2,0)}*+{\F_{1,a}/\p^1_a}="F1a"
!{(2,0)}*+{\F_{1,b}/\p^1_b}="F1b"
!{(0,1)}*+{\p^2/\pt}="P2"
"T"-@{->}"Sa"-@{->}"F0a"
"T"-@{->}"Sb"-@{->}"F0b"
"T"-@{->}"F0p"
"F0p"-@{->}"F0a"
"F0p"-@{->}"F0b"
"T"-@{->}"F1ap"-@{->}"F1a"
"T"-@{->}"F1bp"-@{->}"F1b"
"F1ap"-@{->}"P2"
"F1bp"-@{->}"P2"
"Sa"-@{->}"F1a"
"Sb"-@{->}"F1b"
"T"-@{.>}"F1a"
"T"-@{.>}"F1b"
"T"-@{.>}"F0a"
"T"-@{.>}"F0b"
"T"-@{.>}"P2"
}
\]
\caption{} \label{fig:exple1}
\end{figure}
\end{example}

\begin{figure}[ht]
\[
\mygraph{
!{<0cm,0cm>;<1.4cm,0cm>:<0cm,1.6cm>::}
!{(0,0)}*+{S_{a,b}/\pt}="Sab"
!{(1,1.5)}*+{S_{b,c}/\pt}="Sbc"
!{(-1,1.5)}*+{S_{a,c}/\pt}="Sac"
!{(-1,.5)}*+{\F_{1,a}/\pt}="F1a"
!{(1,.5)}*+{\F_{1,b}/\pt}="F1b"
!{(0,2)}*+{\F_{1,c}/\pt}="F1c"
!{(0,1)}*+{\p^2/\pt}="P2"
"Sab"-@{->}"F1a"
"Sab"-@{->}"F1b"
"Sac"-@{->}"F1a"
"Sac"-@{->}"F1c"
"Sbc"-@{->}"F1c"
"Sbc"-@{->}"F1b"
"Sab"-@{.>}"P2"
"Sac"-@{.>}"P2"
"Sbc"-@{.>}"P2"
"F1a"-@{->}"P2"
"F1b"-@{->}"P2"
"F1c"-@{->}"P2"
}
\]
\caption{} \label{fig:exple2}
\end{figure}

\begin{example} \label{exple:complex 2}
Consider the blow-ups of 3 points $a,b,c$ on $\p^2$.
These gives three squares around the corresponding vertex of type $\p^2$ (see Figure \ref{fig:exple2}).
In particular the square complex $\Xl$ is not CAT$(0)$, as mentioned in the introduction.
\end{example}

\subsection{Sarkisov links and elementary relations}

In this section we show that the complex $\Xl$ encodes the notion of Sarkisov links, and of elementary relation between them.

First we rephrase the usual notion of Sarkisov links between 2-dimensional Mori fiber spaces.
Let $(S/B, \phi)$, $({S}'/B', \phi')$ be two marked rank 1 fibrations.
We say that the induced birational map $S \rat S'$ is a \textit{Sarkisov link} if there exists a marked rank 2 fibration $S''/B''$ that factorizes through both $S/B$ and $S'/B'$.
Equivalently, the vertices corresponding to $S/B$ and $S'/B'$ are at distance 2 in the complex $\Xl$, with middle vertex $S''/B''$:

\[\xymatrix@R-15pt{
& S''/B'' \ar[dl] \ar[dr] \\
S/B && S'/B'
}\]

This definition is in fact equivalent to the usual definition of a link of type I, II, III or IV from $S/B$ to $S'/B'$ (see \cite[Definition 2.14]{Kaloghiros} for the definition in arbitrary dimension).
Below we recall these definitions in the context of surfaces, in terms of commutative diagrams where each morphism has relative Picard number 1 (such a diagram corresponds to a ``two rays game''), and we give some examples.
Remark that these diagrams are not part of the complex $\Xl$: in each case, the corresponding subcomplex of $\Xl$ is just a path of two edges, as described above.

\begin{itemize}[wide]
\item Type I: $B$ is a point, $B' \simeq \p^1$,  and ${S}' \to S$ is the blow-up of a point of degree $d \ge 1$ such that we have a diagram
\[\xymatrix@R-15pt@C-15pt{
& S' \ar[dl] \ar[dr]\\
S \ar[dr] && \p^1 \ar[dl] \\
& \pt
}\]
Then we take $S''/B'' := S'/\pt$.

Examples are given by the blow-up of a point of degree $1$, or a general point of degree $4$, on $S = \p^2$.
The fibration $S'/\p^1$ corresponds respectively to the lines through the point of degree 1, or to the conics through the point of degree 4.

\item Type II: $B = B'$, and there exist two blow-ups $S'' \to S$ and $S'' \to S'$ that fit into a diagram of the form:
\[\xymatrix@R-15pt@C-15pt{
& S'' \ar[dl] \ar[dr]\\
S \ar[dr] && {S}' \ar[dl] \\
& B
}\]
Then we take $S''/B'' := S''/B$.

An example is given by blowing-up a point of degree 2 on $S = \p^2$, and then by contracting the transform of the unique line through this point. 
The resulting surface $S'$ is a del Pezzo surface of degree 8, which has rank 1 over $\K$, but has rank 2 over $\K^a$ (being isomorphic to $\p^1_{\K^a} \times \p^1_{\K^a}$).
Other examples, important for this paper, are provided by blowing-up a point of degree 8 on $\p^2$: see \S\ref{sec:bertini}.

\item Type III: symmetric situation of a link of type I.

\item Type IV: $(S, \phi)$ and $({S}', \phi')$ are equal as marked surfaces, but the fibrations to $B$ and $B'$ are distinct.
In this situation $B$ and $B'$ must be isomorphic to $\p^1$, and we have a diagram
\[\xymatrix@R-15pt@C-15pt{
& S \ar[dl] \ar[dr]\\
B \ar[dr] && B' \ar[dl] \\
& \pt
}\]
Then we take $S''/B'' := S/\pt$.

For rational surfaces, a type IV link always corresponds to the two rulings on $\F_0 = \p^1 \times \p^1$, that is, $S/B = \F_0/\p^1$ is one of the rulings, $S'/B' = \F_0/\p^1$ the other one, and $S''/B'' = \F_0/\pt$.
See \cite[Theorem 2.6 (iv)]{iskovskikh_1996} for other examples in the context of non-rational surfaces.
\end{itemize}

A \textit{path of Sarkisov links} is a finite sequence of marked rank 1 fibrations 
\[(S_0/B_0, \phi_0), \dots, (S_n/B_n, \phi_n),\]
such that for all $0 \le i \le n-1$, the induced map $g_i\colon S_i/B_i \dashrightarrow S_{i+1}/B_{i+1}$ is a Sarkisov link.

\begin{proposition} \label{pro:from S3}
Let $(S'/B, \phi)$ be a marked rank $3$ fibration.
Then there exist finitely many squares in $\Xl$ with $S'$ as a corner, and the union of these squares is a subcomplex of $\Xl$ homeomorphic to a disk with center corresponding to $S'$.
\end{proposition}

\begin{proof}
Since $\rho(S') = 3$ or 4, we can factorize the fibration $S'/B$ into 
\[S' \to S \to Y \to B,\]
where $S' \to S$ is a divisorial contraction, and $Y$ is either a surface or a rational curve.
By playing the two rays game on $S'/Y$ (see Remark \ref{rem:2 rays}), we obtain another surface $\tilde S$ and a divisorial contraction $S' \to \tilde S$ that fits into a commutative diagram:
\[\xymatrix@R-20pt{
& S \ar[dr]\\
S' \ar[ur] \ar[dr] & & Y \ar[r] & B \\
& \tilde S \ar[ur]
}\]
If $Y$ is a surface, we obtain the following square in $\Xl$:
\[\xymatrix@R-10pt@C-10pt{
& S'/B \ar[dl] \ar[dr]\\
S/B \ar[dr] & & \tilde S/B \ar[dl] \\
& Y/B
}\]
On the other hand if $Y \simeq \p^1$ is a curve (and so $B$ is a point), then we obtain the following two squares in $\Xl$:
\[\xymatrix@R-20pt@C-10pt{
& S'/\pt \ar[dl] \ar[dr] \ar[dd]\\
S/\pt \ar[dd] & & \tilde S/\pt \ar[dd] \\
& S'/\p^1 \ar[dl] \ar[dr] \\
S/ \p^1 & & \tilde S/\p^1
}\]
Now in both cases we consider the two rays game on $\tilde S/B$: this produces $\tilde Y$, which is either a surface or a curve, and which fits into a diagram:
\[\xymatrix@R-20pt{
& S \ar[dr]\\
S' \ar[ur] \ar[dr] & & Y \ar[dr] \\
& \tilde S \ar[ur] \ar[dr] & & B \\
& & \tilde Y \ar[ur]
}\]
Then by considering the two rays game on $S'/\tilde Y$, we produce one or two new squares in $\Xl$ that are adjacent to the previous ones.
After finitely many such steps, the process must stop and produce the expected disk, because by Theorem \ref{thm:cone} there are only finitely many divisorial contractions that we can use to factor $S'/B$.
\end{proof}

\begin{remark}
If $B \simeq \p^1$, then at each step $Y$ is a surface, and in this case the disk produced by the proof consists of exactly four squares in $\Xl$.

Precisely, with the notation of the proof, the morphisms $S/Y$ and $\tilde S/Y$
correspond to the blow-ups of points $p$ and $\tilde p$, and $S'$ is the blow-up
of both.
Observe that  $p, \tilde p$ can have arbitrary degrees, but they are not on the same fiber of $Y/\p^1$ because otherwise $S'/\p^1$ would not be a rank 3 fibration (the anticanonical divisor would not be ample because of the presence of $(-2)$-curves).  
Let $E_p, E_{\tilde p}$ be the corresponding exceptional divisors in $S'$, and $F_p, F_{\tilde p}$ the strict transforms of the fibers through $p$ and $\tilde p$ respectively.
Then the four squares correspond to the four possible choices of contraction of two divisors from $S'$ over $\p^1$: either $E_p$ or $F_p$, and independently either $E_{\tilde p}$ or $F_{\tilde p}$. 

On the other hand if $B$ is a point the number of squares might vary:
for instance in Example \ref{exple:complex 1} we saw a situation with five squares.
\end{remark}

In the situation of Proposition \ref{pro:from S3}, by going around the boundary of the disk we obtain a path of Sarkisov links whose composition is the identity (or strictly speaking, an automorphism).
We say that this path is an \textit{elementary relation} between Sarkisov links, coming from ${S'}^3/B$. 
More generally, any composition of Sarkisov links that corresponds to a loop in the complex $\Xl$ is called a \textit{relation} between Sarkisov links.

\section{Relations in the Sarkisov program in dimension 2} \label{sec:sarkisov}

In this section we prove that the complex $\Xl$  is connected and simply connected, which will be the key in proving Theorem \ref{thm:main big amalgam}.
These connectedness results will follow from the Sarkisov program, and more precisely from the study of relations in the Sarkisov program, that we can state as follows:

\begin{theorem}[Sarkisov Program]~ \label{thm:sarkisov}
\begin{enumerate}
\item \label{sarkisov1} Any birational map $f\colon S \rat S'$ between rank $1$ fibrations is a composition of Sarkisov links $($and automorphisms$)$.
\item \label{sarkisov2} Any relation between Sarkisov links is generated by elementary relations.
\end{enumerate}
\end{theorem}

In arbitrary dimension over $\C$, these results correspond to \cite[Theorem 1.1]{HMcK} and \cite[Theorem 1.3]{Kaloghiros}.
We give an account of the proof of these results in the simpler case of surfaces, but working over an arbitrary perfect base field.

\subsection{Polyhedral decomposition}

Let $\{ (S_1/B_1, \phi_1), \dots, (S_m/B_m, \phi_m) \}$ be a finite collection of marked rank 1 fibrations.
In the context of Theorem \ref{thm:sarkisov},  we will take $m = 2$, $S_1 = S$, $S_2 = S'$ when proving \ref{sarkisov1}, or the entire collection of rank 1 fibrations visited by a relation of Sarkisov links when proving \ref{sarkisov2}.

By repetitively applying  Proposition \ref{pro:resolution}, we produce a marked surface $Z$  dominating all the $S_i$, that is, such that
all induced birational maps $f_i\colon Z \to S_i$ are morphisms.
We pick a sufficiently small ample $\Q$-divisor $A$ on $Z$ such that for all $i$, $f_i$ is $(K+A)$-negative, and $-K_{S_i}- {f_i}_*(A)$ is relatively ample over $B_i$.
The point of choosing such an ample divisor $A$ is to ensure that there exist only finitely many $(K+A)$-negative birational morphisms from $Z$ (up to post-composition by an isomorphism).
Indeed this follows from Theorem \ref{thm:cone}, which says that at each step there are only finitely possible divisorial contractions.
If there are only finitely many birational morphisms from $Z$ (for instance if $Z$ is a del Pezzo surface), we also admit the choice $A = 0$.

For each $i = 1, \dots, m$, applying the following two steps we construct an (effective) ample divisor $\Delta_i$ on $Z$ such that
$f_i\colon Z \to S_i$ is $(K+A)$-negative with scaling of $\Delta_i$, and more precisely $f_i$  will be a $(K+A+\Delta_i)$-ample model
of $Z$:
\begin{enumerate}[wide]
\item If $B_i \simeq \p^1$, pick $G_i$ a large multiple of the fiber of $S_i/B_i$ such that $-K_{S_i}- {f_i}_*(A) + G_i$ is ample on $S_i$.
If $B_i = \{pt\}$, then  $-K_{S_i}- {f_i}_*(A)$ is already ample, so we just set $G_i = 0$.
In both cases $G_i$ is a nef divisor on $S_i$.
\item Now pick $P_i$ an effective $\Q$-divisor on $S_i$, equivalent to the ample divisor $-K_{S_i}- {f_i}_*(A) + G_i$, and set $\Delta_i = A + f_i^*(P_i)$, which is ample as the sum of ample and nef divisors.
One checks that 
\begin{align*}
{f_i}_*(K + A + \Delta_i) &= K_{S_i} + {f_i}_*(A) + {f_i}_*(A) - K_{S_i} - {f_i}_*(A) + G_i \\
&=  {f_i}_*(A) + G_i,
\end{align*}
which is an ample divisor on $S_i$ as expected.
\end{enumerate}

We can assume that the $\Delta_i$ generate $N^1(Z)$ (throw in more ample divisors if necessary).
We choose some rationals $r_i > 0$ such that $K + A + r_i\Delta_i$ is ample.
We say that a $\Q$-divisor $\Delta$ is a \textit{subconvex combination} of the $r_i\Delta_i$ if 
\[\Delta = \sum t_i r_i \Delta_i \text{ with } t_i \ge 0, \sum t_i \le 1.\]
Let $\{g_j\}_{j=1}^{s}$ be the finite collection of $(K+A+\Delta)$-negative birational morphisms (up to isomorphism on the range), for all choices of a subconvex combination $\Delta$ as above.
Observe that this collection is indeed finite because of Theorem~\ref{thm:cone}.
Remark also that the initial $f_i$, $i = 1, \dots, m$, are part of this collection by construction, and also the identity $\id\colon Z \to Z$ (corresponding for instance to taking one of the $t_i$ equal to 1, and all the other equal to 0, because then $D = K+A+r_i\Delta_i$ is ample by assumption, and the corresponding embedding $\phi_D$ is equivalent to the identity morphism).

For any family $\{D_i\}_{i\in I}\in N^1(Z)$ of divisors in the vector space $N^1(Z)$, we denote by $\ConeConv(D_i\mid i\in I)$ the cone over the convex hull of the $D_i$.
Now we intersect the cone 
\[\ConeConv(K+A, K + A + r_1\Delta_1, \dots, K+ A + r_s\Delta_s),\] 
with the pseudo-effective cone $\overline{\Eff}(Z)$ to get a cone $\Cl^\circ$.
Explicitly:
\[\Cl^\circ = \left\lbrace
D = \lambda \left( K + A + \sum t_i r_i \Delta_i \right);\; \lambda,t_i \ge 0,
\sum t_i \le 1, D \text{ pseudo-effective}
\right\rbrace \subseteq N^1(Z).\]
We denote by $\Cl$ the intersection of $\Cl^\circ$ with an affine hyperplane defined by $K+A$:
\[\Cl := \{ D \in \Cl^\circ;\; (K+A)\cdot D = -1 \}.\]

Recall from Proposition \ref{pro:factorization} that each $g_j\colon Z \to S_j$ is a finite sequence of contractions of exceptional divisors, that is, orbits under $\Gal(\K^a/\K)$ of pairwise disjoint $(-1)$-curves. 
We denote $\{C_i\}_{i \in I}$ the finite collection of classes in $N^1(Z)$ obtained as pull-back of such exceptional divisors contracted by $g_j$, for all $j = 1, \dots, s$.
We introduce a notation for the hyperplane and half-spaces defined by $C_i$ in $N^1(Z)$:
\begin{align*}
C_i^\perp &= \{ D \in N^1(Z); D\cdot C_i = 0\},\\
C_i^\ge &= \{ D \in N^1(Z); D\cdot C_i \ge 0\},
\end{align*}
and similarly for $C_i^{>}$, $C_i^\le$, $C_i^{<}$.

Let $D \in \Cl$ be a big divisor.
Then we get a partition $I = I^+ \cup I^-$ such that 
\[D \in \bigcap_{i \in I^+} C_i^{>}  \cap \bigcap_{i \in I^-} C_i^{\le}.\]
There exists $j$ such that the morphism $\phi_D$ associated with $D$ (see discussion after Theorem \ref{thm:BPF}) coincides with $g_j:Z \to S_j$.
The classes $C_i$, $i \in I^-$, correspond to the curves contracted by $\phi_D =
g_j$.

\begin{lemma} \label{lem:polyhedral}
The cone $\Cl^\circ \subset N^1(Z)$ is rational polyhedral and convex, hence also the affine section $\Cl$.
\end{lemma}

\begin{proof}
First we prove that  
\[\nef(Z) \cap \Cl^\circ =  \bigcap_{i \in I} C_i^{\ge} \cap \ConeConv(K+A, K + A +  r_1\Delta_1, \dots, K+ A + r_s\Delta_s),\]
from which it follows that $\nef(Z) \cap \Cl^\circ$ is a rational polyhedral convex cone.

If $D = K + A + \Delta \in \Cl^\circ$ is not nef, then by construction of $\phi_D$ there exists $i \in I$ such that the exceptional divisor $C_i$ is contracted by $\phi_D$, and $D \cdot C_i < 0$. 
On the other hand if $D \in \Cl^\circ$ is nef, then for any irreducible component $C$ in the support of one of the $C_i$ we have $D \cdot C \ge 0$, hence $D \cdot C_i \ge 0$ for all $i$.

Now we show that $\Cl^\circ$ is a rational polyhedral convex cone.
Let $D \in \Cl^\circ$ be a big divisor, and $\phi_D\colon Z \to S$ the associated birational morphism.
Up to reordering the $C_i$, we can assume that $C_1, \dots, C_r$ are the classes contracted by $\phi_D$, where $r$ is the relative Picard number of $Z$ over $S$.
Then we can write
\[D = {\phi_D}^* ({\phi_D}_*(D)) + \sum_{i=1}^r a_i C_i,\]
where ${\phi_D}^* ({\phi_D}_*(D)) \in \nef(Z) \cap \Cl^\circ$, and the $a_i$ are positive by Lemma \ref{lem:negative}.
Together with $K+A+r_i\Delta_i \in \nef(Z)\cap \Cl^\circ$ this gives
\begin{align*}
\Cl^\circ &\subseteq  \ConeConv(\nef(Z)\cap\Cl^\circ,K+A) \bigcap \ConeConv(\nef(Z) \cap \Cl^\circ,C_i\mid i\in I) 
\end{align*}
This inclusion is in fact an equality because the first cone on the right-hand side is just $\ConeConv(K+A,K+A+r_1\Delta_1,\dots,K+A+r_s\Delta_s)$ and the second cone consists of pseudo-effective divisors, so the right hand side is contained in $\Cl^\circ$.
It follows that $\Cl^\circ$ is rational polyhedral and convex, as expected.
\end{proof}

We set
\[\Al_j := \bigcap_{i \in I^+} C_i^{>}  \cap \bigcap_{i \in I^-} C_i^{\le} \cap \Cl.\]
In particular, $\Al_j$ is a rational polyhedral subset of $\Cl$, and is equal to the set of divisors $D \in \Cl$ such that $\phi_D = g_j$.  
Observe that the chamber $\Al$ of ample divisors in $\Cl$ is one of the $\Al_j$, associated to $g_j = \id$, and to the partition $I^+ = I$, $I^- = \emptyset$. 
Clearly the $\Al_j$ form a partition of the interior of $\Cl$.

We just reproved \cite[Theorem 4.2]{KKL} (which is stated in arbitrary dimension, but over $\C$):

\begin{theorem} \label{thm:polyhedral chambers}
The interior of the cone $\Cl$ admits a finite partition in polyhedral chambers $\Int(\Cl) = \bigcup_{j} \Al_j$. 
\end{theorem}

For further reference we sum-up the above discussion:
\begin{setup} \phantomsection \label{setup:C}\item
\begin{itemize}[wide]
\item We start with a finite collection $\{ (S_1/B_1, \phi_1), \dots, (S_m/B_m, \phi_m) \}$ of marked rank 1 fibrations.
\item We pick $Z$ a common resolution with Picard number $\rho(Z) \ge 4$.
\item We choose $A$ an ample $\Q$-divisor on $Z$ such that each map $Z \to S_i$ is $(K+A)$-negative with scaling of an ample divisor $\Delta_i$.
If there are finitely many birational morphisms from $Z$, we allow $A = 0$.
\item We construct a convex cone $\Cl^\circ$ in $N^1(Z)$, by considering the union of all segments $[\Delta, K+A] \cap \overline{\Eff}(Z)$, for all  convex combinations $\Delta$ of the ample divisors $K + A + r_i\Delta_i$, and by taking the cone  over these. 
In practice, we work with $\Cl$, the section of $\Cl^{\circ}$ by the affine hyperplane corresponding to classes $D$ such that $(K +A )\cdot D = -1$. 
\item Each class $D$ in the interior of $\Cl$ corresponds to a $(K+A)$-birational morphism $\phi_D\colon Z \to S_D$. 
\item Conversely, given $g_j\colon Z \to S_j$ a $(K+A)$-negative birational morphism, the set $\Al_j$ of divisors $D$ in $\Cl$ such that $g_j = \phi_D$ form a polyhedral chamber $\Al_j$ with non-empty interior.
\end{itemize}
\end{setup}

\begin{remark}
In higher dimension several complications arise that we avoided in the above discussion. 
In particular in dimension $\ge 3$ it is not true anymore that each $\Al_j$ spans $N^1(Z)$, because of the appearance of small contractions.

We should also mention that in dimension 2, the decomposition of Theorem \ref{thm:polyhedral chambers} can  be
phrased in terms of Zariski decompositions (see \cite{BKS}).
Namely, in each chamber $\Al_j$ the support of the negative part of the Zariski decomposition is constant, and corresponds to the support of the classes $C_i$ with $i \in I^-$.

Finally, as already noticed in Remark \ref{rem:phi D}, since we only consider adjoint divisors of the form $K + \Delta$ with $\Delta$ ample, we can work directly in the Néron-Severi space $N^1(Z)$ instead of choosing a subspace of the space of Weil divisors, as in \cite{HMcK, Kaloghiros}.
\end{remark}

\begin{example} \label{exple:blowup 2 points}
In \cite[Figures 1 and 6]{Kaloghiros}  the above construction is illustrated by the case of $\p^2$ blown up at 2 or 3 distinct points.
Here we consider the case of $\p^2$ blown up at two points (of degree 1), with one infinitely near the other.
Precisely, let $Z$ be the surface obtained from $\mathbb{P}^2$ by blowing-up a point $p = L \cap L'$ intersection of two lines, producing an exceptional divisor $E$, and then $p' = E \cap L'$, producing an exceptional divisor $E'$ (see Figure \ref{fig:blowups}, where the numbers in bracket denote self-intersection, and where we use the same notation for a curve and its strict transforms).

\begin{figure}[ht]
\def\svgwidth{.8\textwidth}
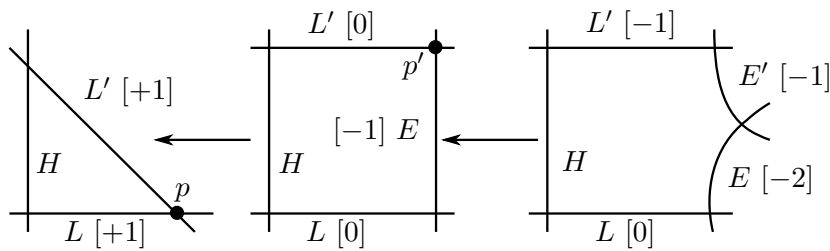
\caption{$\p^2$ blown up at $p$ and $p'$.} \label{fig:blowups}
\end{figure}

On $Z$, the curves $E$, $E'$ and $L'$ are the only irreducible divisors with negative self-intersection, and they generate the pseudo-effective cone $\overline{\Eff}(S)$.
We also denote by $H$ the class of a generic line from $\p^2$, and as usual $K$ is the canonical divisor.
On $Z$ we have 
\begin{align*}
H &= L'+E+2E', &
-K &= 3L' + 2E + 4E'.
\end{align*}
The classes $C_i$ of contracted curves are
\begin{align*}
C_1 &= L', &
C_2 &= E', &
C_3 &= E + E'.
\end{align*}

In this simple example, it turns out that any class $D$ in the pseudo-effective cone $\overline{\Eff}(Z)$ corresponds to a birational morphism $\phi_D$.
There are 6 possibilities for this morphism $\phi_D$, and Figure~\ref{fig:cone 2 pts} shows the chamber decomposition for the whole $\overline{\Eff}(Z)$.

\begin{figure}[ht]
\[
\xygraph{
!{<0cm,0cm>;<3cm,0cm>:<0cm,1.8cm>::}
!{(-0.5,0)}*{\bullet}="L'"
!{(0.5,0)}*{\bullet}="L"
!{(.5,2)}*{\bullet}="E"
!{(.8,.8)}*{\bullet}="ell"
!{(1.5,0)}*{\bullet}="E'"
!{(1,1)}*{\bullet}="E+E'"
!{(0.5,0.6)}*{\bullet}="fantome"
!{(0.25,1)}*{\Al^s_{E,L'}}
!{(0.62,1.1)}*{\Al^s_{E}}
"L"-"ell" "L"-"E'"
"ell"-^<{H}"E'" "fantome"-"L'" "L"-_>{E}"E"
"L"-^<(-0.1){L}"L'" "L'"-"E'"
"E'"-_<{E'}_>{E + E'}"E+E'" "E+E'"-"E"
"L'"-^<{L'}"E"
"ell"-"E"
"ell"-|\bullet_{-K}"fantome"
}
\]
\caption{} \label{fig:cone 2 pts}
\end{figure}

However, Set-Up \ref{setup:C} only guaranties the chamber decomposition for classes in $\Cl$, which is pictured on Figure \ref{fig:cone vert} (where we work with the choice $A = 0$).
Observe in particular that chambers $\Al^s_{E}$ and $\Al^s_{E,L'}$ in Figure \ref{fig:cone 2 pts} corresponds to singular surfaces, namely the blow-down of $E$, or the blow-down of $E,L'$.
Equivalently, a wall between chambers in Figure \ref{fig:cone 2 pts} does not always correspond to the contraction of a $(-1)$-curve.
In contrast, chambers $\Al_1, \dots, \Al_4$ in Figure \ref{fig:cone vert} correspond to smooth surfaces, and the chambers are delimited by the
hyperplanes $C_i^\perp$.

\begin{figure}[ht]
\def\svgwidth{.6\textwidth}
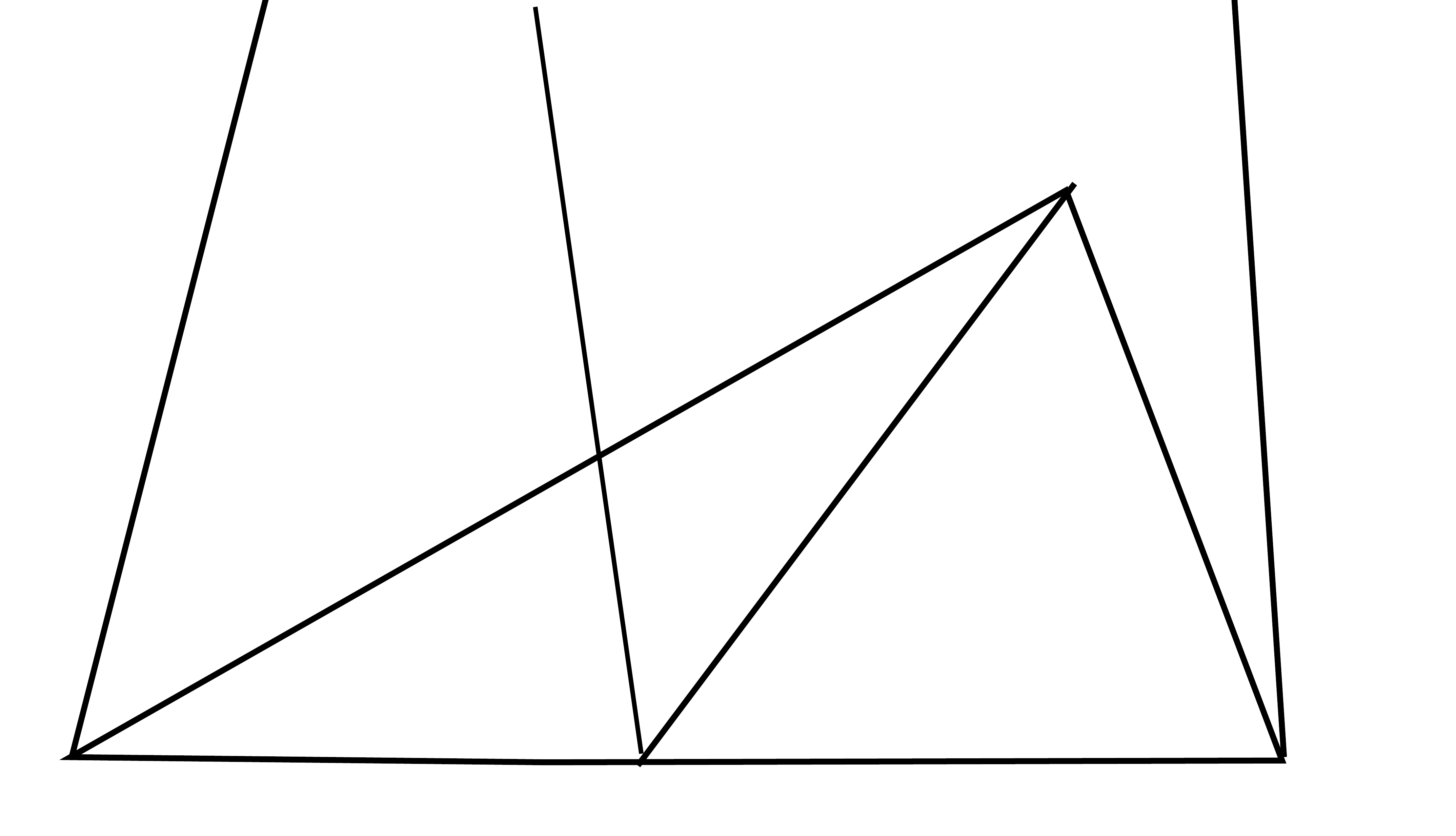
\caption{The partition $\Int(\Cl)=\bigcup_{i=1}^4\Al_i$ (in gray).} 
\label{fig:cone vert}
\end{figure}
\end{example}

The following two propositions correspond to \cite[Theorem 3.3]{HMcK}, in the case of surfaces.

\begin{proposition} \label{pro:factorisation1}
Assume Set-Up $\ref{setup:C}$, let $\Al_j$ be one of the polyhedral chambers given by Theorem~$\ref{thm:polyhedral chambers}$, and $g_j\colon Z \to S_j$ the associated birational morphism.
If $D \in \widebar{\Al_j} \setminus \Al_j$ with associated morphisms $\phi_D\colon Z \to Y$, then there exists a unique morphism $f\colon S_j \to Y$ such that $\phi_D = f \circ g_j$.
Moreover the morphism $\phi_D$ only depends on the face $\Fl^r \subset \widebar{\Al_j}$ such that $D$ is in the relative interior of $\Fl^r$.
\end{proposition}

\begin{proof}
Set $D_0 = D \in \widebar{\Al_j} \setminus \Al_j$, and pick $D_1 \in \Al_j$.
Then for all $t \in (0,1]$, we have $D_t := tD_1 + (1-t)D_0 \in \Al_j$.
The curves contracted by $g_j = \phi_{D_t}$ are the curves $C$ such that $D_t \cdot C \le 0$ for $ t \in (0,1]$, and so also for $t = 0$.
Hence $\NE(\phi_D) \subset \NE(g_j)$, and the Rigidity Lemma below gives the expected morphism $f\colon S_j \to Y$. 

Finally the curves contracted by $f$ are the curves $C$ on $S_j$ such that $(g_j)_* D \cdot C = 0$, and these conditions on $D$ define the minimal face $\Fl^r$ containing $D$.  
\end{proof}

\begin{lemma}[Rigidity Lemma, see {\cite[Proposition 1.14]{Debarre}}]
\label{lem:rigidity}
Let $X,Y,Y'$ be projective varieties and let $\pi\colon X \to Y$, $\pi'\colon X \to Y'$ be morphisms with connected fibers.
If $\NE(\pi) \subset \NE(\pi')$, there is a unique morphism $f\colon Y \to Y'$ with connected fibers such that $\pi' = f \circ \pi$.
\end{lemma}

\begin{proposition} \label{pro:factorisation2}
Assume Set-Up $\ref{setup:C}$.
Let $j,k$ be two indexes such that $\widebar{\Al_j} \cap \Al_k \neq \emptyset$, and let $g_{j,k}\colon S_j \to S_k$ be the morphism $($given by Proposition $\ref{pro:factorisation1})$ such that $g_k = g_{j,k} \circ g_j$.
Then the relative Picard number of $g_{j,k}\colon S_j \to S_k$ is equal to the codimension of $\widebar{\Al_j} \cap \widebar{\Al_k}$ in $\widebar{\Al_j}$.
\end{proposition}

\begin{proof}
By definition there exist two partitions of the set of indexes $I = I_j^+ \cup I_j^- = I_k^+ \cup I_k^-$ such that
\begin{align*}
\Al_j &= \bigcap_{i \in I_j^+} C_i^{>}  \cap \bigcap_{i \in I_j^-} C_i^{\le}
\cap \Cl, \\
\Al_k &= \bigcap_{i \in I_k^+} C_i^{>}  \cap \bigcap_{i \in I_k^-} C_i^{\le}
\cap \Cl.
\end{align*}
The condition  $\widebar{\Al_j} \cap \Al_k \neq \emptyset$ means that $I_j^- \subsetneq I_k^-$.
Let $\{ i_1, \dots, i_t \} = I_j^+ \cap I_k^-$.
Then 
\[\widebar{\Al_j} \cap \widebar{\Al_k} =
\bigcap_{i \in I_k^+} C_i^{\ge} \cap
\bigcap_{i \in \{ i_1, \dots, i_t \}} C_i^\perp \cap
\bigcap_{i \in I_j^-} C_i^{\le} \cap \Cl.\]
The morphism $g_{j,k}$ corresponds to the contraction of the classes $C_{i_1}, \dots, C_{i_t}$, and $t$ is by construction the codimension of $\widebar{\Al_j} \cap \widebar{\Al_k}$ in $\widebar{\Al_j}$.
\end{proof}

\subsection{Boundary of \texorpdfstring{$\Cl$}{C}}

Assuming Set-up \ref{setup:C}, we define $\partial^+ \Cl$ as the intersection of $\Cl$ with the boundary of the pseudo-effective cone in $N^1(Z)$, or equivalently, as the classes in $\Cl$ that are not big.
If $\rho = \rho(Z)$ is the Picard number of $Z$, then by assumption $\Cl^{\circ}$ is a cone of full dimension $\rho$, hence the affine section $\Cl$ is homeomorphic to a ball $B^{\rho-1}$ and $\partial^+ \Cl$ is homeomorphic either to a ball $B^{\rho-2}$, or to a sphere
$S^{\rho-2}$, depending whether  the pseudo-effective cone is contained in
$\Cl$ or not.
For instance the sphere situation arises if $Z$ is del Pezzo, $A = 0$, and the initial collection of rank 1 fibrations $S_i/B_i$ is the full (finite) collection of such fibrations dominated by $Z$.

Now we put a structure of polyhedral complex on $\partial^+ \Cl$.
We consider the set of chambers $\Al_j$ such that $\widebar{\Al_j} \cap \partial^+ \Cl$ contains a codimension 1 face $\Wl$ of the closed polytope $\widebar{\Al_j}$.
We call such a face $\Wl$ a \textit{window} of $\Al_j$.
Then the collection of such $\Wl$ uniquely defines a structure of polyhedral complex on $\partial^+ \Cl$, such that the $\Wl$ are the maximal faces of the complex.
More generally we shall denote $\Fl^{r}$ a (closed) codimension $r$ face in $\partial^+ \Cl$.
Here the codimension is taken relatively to the ambiant space $N^1(Z)$, in particular windows correspond to codimension $1$ faces.
We say that $\Fl^{r}$ is an \textit{inner} face if it intersects the relative interior of $\partial^+ \Cl$.
Equivalently, $\Fl^{r}$ is inner if it can be written as the intersection of $r$ windows.

For a window $\Wl$ of the chamber $\Al_j$, $D$ in the relative interior of  $\Wl$, and sufficiently small $ \eps > 0$, the divisor $D' := D - \eps (K+A)$ is in $\Al_j$, and the images $S = \phi_{D'}(Z)$ and $B = \phi_{D}(Z)$ correspond to a Mori fibration $S/B$ that depends
only on $\Wl$ (and not on the particular choice of $D$ or $\eps$).

We see that the codimension $1$ faces in $\partial^+ \Cl$ are in bijection with the rank 1 fibrations $S^1/B$ dominated by a $(K+A)$-negative map from $Z$, and such that $S^1/B$ is $(-K-A)$-ample. 
More generally, we have:

\begin{proposition} \phantomsection \label{pro:F^k}
\begin{enumerate}[wide]
\item Let $\Fl^{r}$ be an inner codimension $r$ face in the polyhedral complex $\partial^+ \Cl$.
Then there exists a rank $r$ fibration $S^{r}/B$ such that:
\begin{itemize}[wide]
\item The induced map $Z \to S^{r}$ is equal to $g_j$ for some $j \in \{1, \cdots, s\}$, in particular this is a $(K+A)$-negative birational morphism;
\item The chamber $\Al_j$ associated with $g_j$ satisfies $\widebar{\Al_j} \supseteq \Fl^{r}$.
\end{itemize}
\item If moreover $\Fl^{r'} \subset \Fl^{r}$ is a strictly smaller face, then the rank $r'$ fibration associated to $\Fl^{r'}$ factorizes through the rank $r$ fibration associated to $\Fl^{r}$. 
\end{enumerate}
\end{proposition}

\begin{proof}
\begin{enumerate}[wide]
\item
Let $D$ be a class in the relative interior of $\Fl^{r}$.
By definition, there exists an ample class $\Delta \in \Cl^\circ$ such that $D$ lies in the segment $[\Delta, K+A]$.
Moreover for sufficiently small $\eps$, $D' = D - \eps(K+A)$ lies in a chamber $\Al_j$ (in fact in the interior of $\Al_j$, since $K+A$ is negative against the exceptional curves $C_i$), where $j$ does not depend on $D$ nor $\eps$.
Let $Z \to B$ be the morphism associated to $D$, and $Z \to S = S_j$ the morphism associated to $D - \eps(K+A)$.
The induced morphism (Proposition \ref{pro:factorisation1}) $S \to B$ 
has connected fibers and contracts curves which are trivial against a divisor of the form $K_S + $ ample, in particular $-K_S$ is relatively ample.
We deduce that $S/B$ is a rank $r'$ fibration for some $r' \ge 1$, and we need to show that $r' = r$. 
By running a relative MMP over $B$, we obtain a factorization of one of the following forms:
\begin{align}
S = S^{r'} &\to S^{r'-1} \to \dots \to S^1 \to B = \p^1 \text{ or } \pt, \tag{$\dagger$} \label{eq:dagger}\\
S = S^{r} &\to S^{r'-1} \to \dots \to S^2 \to \p^1 \to B = \pt.\tag{$\ddagger$} \label{eq:ddagger}
\end{align}
where each $S^{i+1} \to S^i$ is the contraction of one exceptional curve.
In Case (\ref{eq:dagger}), by Proposition \ref{pro:factorisation2} $\Al_{j}$  and $\Al_{S^1}$ share a codimension $(r'-1)$ face, and the intersection of this face with the codimension 1 face $\Fl^1 \subset \partial^+ \Cl$ corresponding to $S^1/B$ gives a codimension $r'$ face $\Fl^{r'}$.
The morphisms associated to $\Fl^{r}\subseteq\widebar{\Al_j}$ and to $\Fl^{r'}\subseteq\widebar{\Al_j}$ are the same, equal to $S/B$, so $\Fl^{r} = \Fl^{r'}$  which gives $r = r'$ as expected.

In Case (\ref{eq:ddagger}), $\Al_{j}$  and $\Al_{S^2}$ share a codimension $(r'-2)$ face $\Fl^{r'-2}$.
Moreover the surface $S^2$ is a Hirzebruch surface and a del Pezzo surface, so $S^2$ is isomorphic to $\F_1$ or $\F_0$.
If $S^2 \simeq \F_1$, then using the factorization $\F_1 \to \p^2 \to \pt$ we are reduced to the previous case.
Finally if $S^2 \simeq \F_0$, the face corresponding to the point is defined by $D\cdot C_1 = 0$ and $D\cdot C_2 = 0$ for $C_1, C_2$ the two rulings of $\F_0$, hence it is a codimension 2 face in $\partial^+ \Cl$.
Intersecting with $\Fl^{r'-2}$ we obtain a codimension $r'$ face in $\partial^+ \Cl$ and conclude as in the previous case that $r' = r$.

\item 
By the previous point, there exist chambers $\Al_j$, $\Al_{j'}$ such that $\Fl^r \subset \widebar{\Al_j}$ and $\Fl^{r'} \subset \widebar{\Al_{j'}}$, with corresponding rank $r$ and $r'$ fibrations $S_j/B$ and $S_{j'}/B'$. 
Moreover the inclusions $\Fl^{r'} \subset \Fl^r \subset \widebar{\Al_j}$ implies that we also have a morphism $S_j/B'$, which induces a morphism $B/B'$:
\[
\xymatrix{
S_{j'} \ar[d] & S_j \ar[ld] \ar[d] \\
B' & B \ar[l]
}
\]
We need to show that the birational map $S_{j'}\to S_j$ induced by the maps $Z\to S_{j'}$ and $Z\to S_j$ is a morphism. 
By the proof of the previous point, there exists an ample divisor $\Delta'$ and real numbers $t_0 > t_1 > 0$ such that $(1-t)\Delta' + t(K+A) \in \Al_{j'}$ for all $t \in [t_0,t_1)$, and $D' = (1-t_1)\Delta' + t_1(K+A) \in \Fl^{r'}$.
Since any ball around $D'$ meets the face $\Fl^r$ hence also the chamber $\Al_j$, there exists a small perturbation $\Delta$ of $\Delta'$ such that  
\begin{align*}
(1-t_0)\Delta + t_0(K+A) \in \Al_{j'} && 
\text{and} &&
(1-t_1)\Delta + t_1(K+A) \in \Al_{j}.
\end{align*}

As explained after Theorem \ref{thm:BPF}, moving along the segment $[\Delta, K+A]$ corresponds to running a $(K+A)$-MMP with scaling, and in particular this gives the expected birational morphism $S_{j'} \to S_j$.

\qedhere
\end{enumerate}
\end{proof}

\begin{corollary} \label{cor:window and sarkisov}
If the intersection $\Wl_i \cap \Wl_j$ of two windows has codimension $1$ in $\Wl_i$ $($hence also in $\Wl_j)$, then there is a Sarkisov link between the corresponding Mori fiber spaces.
\end{corollary}

\begin{proof}
By Proposition \ref{pro:F^k}, there exists a rank $2$ fibration corresponding to the codimension 2 face $\Fl^2 = \Wl_i \cap \Wl_j$, that factorizes through the rank 1 fibrations associated respectively to $\Wl_i$ and $\Wl_j$.
This is exactly our definition of a Sarkisov link.
\end{proof}

\begin{remark}
The above corollary corresponds to \cite[Lemma 3.17]{Kaloghiros}, but the situation for surfaces is simpler (see Figure \ref{fig:4 links}).
Let $\Al_i$, $\Al_j$ be the two chambers with window $\Wl_i$, $\Wl_j$, and $S_i/B_i$, $S_j/B_j$ the corresponding rank 1 fibrations.
We write $\Fl^2 = \Wl_i \cap \Wl_j$.
We distinguish 3 cases in term of the codimension of the intersection $\widebar{\Al_i} \cap \widebar{\Al_j}$.

\begin{enumerate}[$(a)$,wide]
\item If $\widebar{\Al_i} = \widebar{\Al_j}$, that is, $\Wl_i$ and $\Wl_j$ are two windows of a same chamber, then we have a link of type IV.
In the case of a rational surface, the only possibility is the change of ruling on $\F_0 = \p^1 \times \p^1$; the fibrations $S_i/B_i$ and $S_j/B_j$ correspond to the two rulings on $\F_0$, and the codimension $2$ face $\Fl^2$ to the rank 2 fibration $\F_0/\pt$.

\item If $\widebar{\Al_i}$ and $\widebar{\Al_j}$ share a codimension 1 face, then we have a link of type I or III.
Up to reordering we can assume that we have a blow-up map $S_j \to S_i$.
Then $B_j = \p^1$, $B_i = \pt$, and $\Fl^2$ corresponds to the rank $2$ fibration $S_j/\pt$.

\item Otherwise, let $S^2/B$ be the rank $2$ fibration associated to $\Fl^2$, as given by Proposition \ref{pro:F^k}.
The chamber $\Al_{S^2}$ is distinct from $\Al_i$ and $\Al_j$, otherwise we would be in one of the two previous cases.
Then $\Al_{S^2}$ shares a codimension 1 face with $\widebar{\Al_i}$ and with $\widebar{\Al_j}$, and we have a link of type II. 
\qedhere
\end{enumerate}
\end{remark}

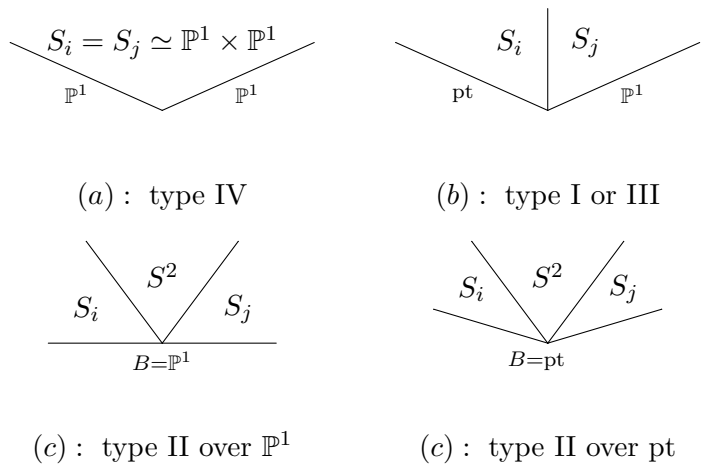
\begin{figure}[ht]
\[
\xymatrix@R-20pt{
\mygraph{
!{<0cm,0cm>;<1cm,0cm>:<0cm,.9cm>::}
!{(0,0)}*{}="0"
!{(-2,1)}*{}="1"
!{(2,1)}*{}="2"
!{(0,1)}*+{S_i = S_j  \simeq \p^1 \times \p^1}
"1"-_{\p^1}"0"-_{\p^1}"2"
}
&
\mygraph{
!{<0cm,0cm>;<1cm,0cm>:<0cm,.9cm>::}
!{(0,0)}*{}="0"
!{(-2,1)}*{}="1"
!{(2,1)}*{}="2"
!{(0,1.5)}*{}="3"
!{(-.5,1)}*+{S_i}
!{(.5,1)}*+{S_j}
"1"-_{\pt}"0"-_{\p^1}"2"
"3"-"0"
}
\\
(a): \text{ type IV}  & (b): \text{ type I  or III}  \\
\
\\
\mygraph{
!{<0cm,0cm>;<1cm,0cm>:<0cm,.9cm>::}
!{(0,0)}*{}="0"
!{(-1.5,0)}*{}="1"
!{(1.5,0)}*{}="2"
!{(-1,1.5)}*{}="3"
!{(1,1.5)}*{}="4"
!{(-1,.5)}*+{S_i}
!{(1,.5)}*+{S_j}
!{(0,1)}*+{S^2}
"1"-_{B = \p^1}"2"
"3"-"0" "4"-"0"
}
& 
\mygraph{
!{<0cm,0cm>;<1cm,0cm>:<0cm,.9cm>::}
!{(0,0)}*{}="0"
!{(-1.5,0.5)}*{}="1"
!{(1.5,0.5)}*{}="2"
!{(-1,1.5)}*{}="3"
!{(1,1.5)}*{}="4"
!{(-1,.8)}*+{S_i}
!{(1,.8)}*+{S_j}
!{(0,1)}*+{S^2}
"1"-_>{B = \pt }"0"-"2"
"3"-"0" "4"-"0"
}
\\
(c): \text{ type II over } \p^1 & (c): \text{ type II over } \pt
}
\]
\caption{Adjacent windows and Sarkisov links} \label{fig:4 links}
\end{figure}

\begin{corollary} \label{cor:relations}
Let $\Fl^3$ be an inner codimension $3$ face in $\partial^+\Cl$.
Let $S^3/B$ be the associated rank $3$ fibration, as given by Proposition $\ref{pro:F^k}$.
Then the elementary relation associated to $S^3/B$ correspond to the finite collection of windows $\Wl_1, \dots, \Wl_{m}$ containing $\Fl^3$ in their closure, and ordered such that $\Wl_j$ and $\Wl_{j+1}$ share a codimension $1$ face for all $j$ $($where indexes are in $\Z/k\Z)$.
\end{corollary}

The following two propositions correspond to the two assertions in Theorem \ref{thm:sarkisov}.

\begin{proposition}\label{pro:connected}
Any birational map $f\colon S \rat S'$ between rank $1$ fibrations is a composition of Sarkisov links, and in particular the complex
$\Xl$ is connected.
\end{proposition}

\begin{proof}
We want to prove that two vertices in $\Xl$ corresponding to two rank 1 fibrations are connected by a path.
Let $S_1/B_1$ and $S_2/B_2$ be these two fibrations, and apply Set-Up \ref{setup:C} to this collection of two fibrations.
Let $\Delta_1$, $\Delta_2$ be ample divisors in $\Cl$ such that for $i = 1,2$, the fibration $S_i/B_i$ corresponds to a $(K+A)$-negative map with scaling of $\Delta_i$.
Up to a small perturbation of $\Delta_1$ and $\Delta_2$, we can assume that the 2 dimensional affine plane containing $\Delta_1, \Delta_2$ and $K+A$ intersects transversally the faces of $\partial^+ \Cl$. 
This means that the windows $\Wl_1$ and $\Wl_2$ corresponding to $S_1/B_1$ and $S_2/B_2$ are connected by a finite sequence of windows, where two successive windows share a codimension 1 face. 
By Corollary \ref{cor:window and sarkisov}, this corresponds to a sequence of Sarkisov links, hence the expected path in the complex $\Xl$.
\end{proof}

\begin{proposition}\label{pro:simply connected}
Any relation between Sarkisov links is a composition of conjugates of elementary relations, and in particular the complex $\Xl$ is simply connected.
\end{proposition}

\begin{proof}
Let $\gamma$ be a loop in $\Xl$.
Without loss in generality, we can assume that the image of $\gamma$ lies in the 1-skeleton of $\Xl$.
We can also assume that the loop visits only vertices of rank 1 or 2.
Indeed, using elementary relations, that is, moving around the boundary of disks as provided by Proposition \ref{pro:from S3}, we can avoid all vertices of rank 3.
Moreover, we can assume that the base point is a vertex of rank 1, $S_1/B_1$.
Now the sequence of vertices of rank 1 visited by the loop corresponds to a sequence $S_1/B_1, S_2/B_2, \dots, S_{m}/B_{m} = S_1/B_1$, where for each $i$, the map from $S_i/B_i$ to $S_{i+1}/B_{i+1}$ is a Sarkisov link. 
Consider the surface $Z$ and the complex $\Cl$ associated to this collection, as in Set-Up \ref{setup:C}, and denote as above $\partial^+ \Cl$ the non-big boundary with its induced structure of polyhedral complex.

Since $\partial^+ \Cl$ is homeomorphic to a ball or a sphere of dimension $\rho(Z) - 2 \ge 2$, it is simply connected.
Now we construct a 2-dimensional simplicial complex $\Bl$ embedded in $\partial^+ \Cl$, as follows.
For each face $\Fl^{r} \subset \partial^+ \Cl$ of codimension $k$ with $3 \ge r \ge 1$, we denote by $p(\Fl^{r})$ the barycenter of the vertices of $\Fl^{r}$. 
The $2$-simplexes of $\Bl$ are defined as the convex hulls of the $p(\Fl^{r})$, for each sequence of nested faces $\Fl^{3} \subset \Fl^{2} \subset \Fl^{1}$.
If $\Fl^{3}$ is an inner face (hence also all other faces of the sequence), we say that the corresponding simplex is an inner simplex of $\Bl$.

The complex $\Bl$ is homeomorphic to the barycentric subdivision of the 2-skeleton of the dual cell complex of 
$\partial^+ \Cl$, and so also is simply connected (recall that the 2-skeleton of a simply connected CW-complex is simply connected,  see \cite[Corollary 4.12]{Hatcher}).
 
The inner simplexes form a subcomplex $\Il \subseteq \Bl$, and $\Il$ is a deformation retract of the interior of $\Bl$: this follows from Lemma \ref{lem:70.1 Munkres} below, with $X = \Bl$, $A = \Il$ and $C$ the boundary of $\Bl$.
It follows that $\Il$ also is simply connected.
Moreover, by Proposition \ref{pro:from S3}, each face $\Fl^3$ of codimension $3$ in $\Il$ is the center of a disk whose boundary corresponds to an elementary relation.
By construction, the interiors of these disks are pairwise disjoint, and any triangle of $\Il$ belongs to one of the disks. 
In conclusion, we can perform the required homotopy of our initial loop to a constant loop inside this subcomplex of $\Xl$ by using elementary relations.
\end{proof}

\begin{figure}[ht]
\def\svgwidth{.99\textwidth}
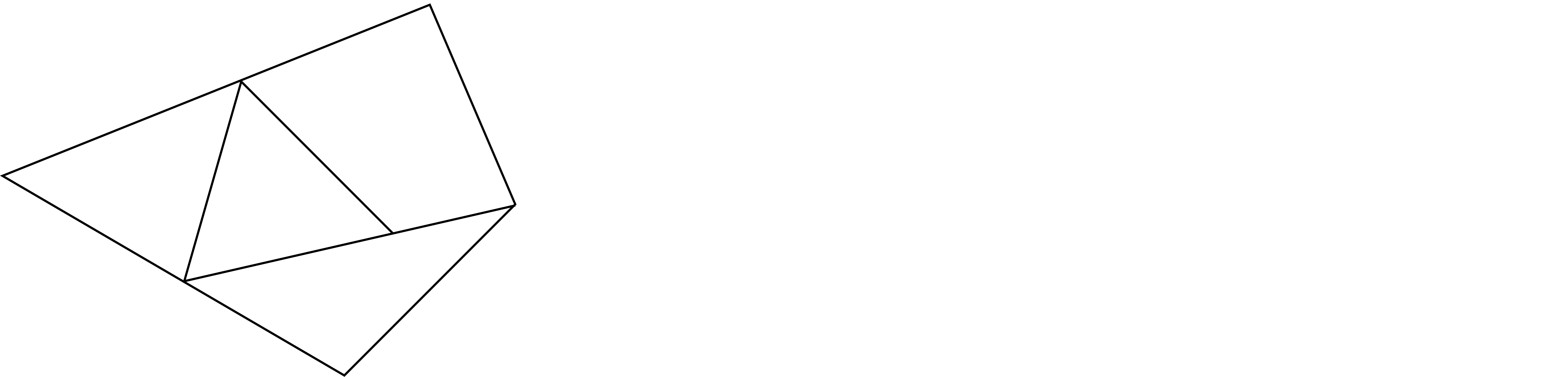
\caption{The complexes $\partial^+\Cl$, $\Bl$ and $\Il$.}
\label{fig:simply connected}
\end{figure}

\begin{lemma}[{\cite[Lemma 70.1]{Munkres}}] \label{lem:70.1 Munkres}
Let $A$ be a full subcomplex of a finite simplicial complex $X$.
Let $C$ consist of all simplices of $X$ that are disjoint from $A$.
Then $A$ is a deformation retract of $X \setminus C$.
\end{lemma}

\section{Elementary generators} \label{sec:generators}

\subsection{Definition}

As in the previous section we consider a  path of Sarkisov links 
\[(S_0/B_0, \phi_0) \dots, (S_n/B_n, \phi_n),\]
that is, for each $0 \le i \le n-1$, there is a Sarkisov link $g_i\colon S_i/B_i \dashrightarrow S_{i+1}/B_{i+1}$.
If in such a path $S_0$, $S_n$ are both isomorphic to $\p^2$, but no intermediate vertex $S_i$ is isomorphic to $\p^2$, we say that
$\phi_n g_{n-1}\cdots g_1 g_0 \phi_0^{-1} \in \Bir_\K(\p^2)$ is an \textit{elementary generator}. 

We denote by $\E$ a choice of representatives of elementary generators in $\Bir_\K(\p^2)$, up to right and left composition by elements of
$\Aut_\K(\p^2)$.
From the Sarkisov program in dimension 2 it directly follows:

\begin{proposition}\label{prop:generators}
Any $f \in \Bir_{\K}(\p^2)$ is a composition of elementary generators, up to an automorphism.
In particular, $\Bir_{\K}(\p^2) = \langle \Aut_\K(\p^2), \E \rangle$.
\end{proposition}

\begin{proof}
By Theorem \ref{thm:sarkisov}\ref{sarkisov1}, there exists a sequence of Sarkisov links decomposing $f$. Then we can cut this sequence at each intermediate surface isomorphic to $\p^2$.
\end{proof}

\begin{remark}
The set $\E$ contains the list of generators given in \cite{IKT}, which contains all Jonqui\`eres transformations.
The set $\E$ is really huge: over an algebraically closed field, in fact every element of $\Bir_\K(\p^2)$ is in $\E$, which does not seem a very clever choice of generators.
However, when working over a non-closed field, and since we never try to work with an explicit presentation of $\Bir_\K(\p^2)$ by generators and relations, the immensity of $\E$ is not a draw-back.
\end{remark}

In the following sections we study two particular examples of elementary generators.

\subsection{Bertini involutions} \label{sec:bertini}

Let $p =\{p_1, \dots, p_8\} \in \p^2$ a point of degree 8.
We say that $p$ is \textit{general}, or equivalently that the $p_i$ are in \textit{general position}, if the blow-up of $p$ is a del Pezzo surface of degree 1, that is, if no line (defined over $\K^a$) contains 3 of the $p_i$, no conic contains six of them, and no cubic is singular at one of them and contains all the others.

Let $S$ be the surface obtained by blowing-up such a general point $p$ of degree 8.
Then $S$ is a rank 2 fibration with exactly two outgoing arrows, and there exists another contraction $S \to \p^2$ that fits into a diagram
\begin{gather} \label{diag:chi8}
\xymatrix{
& S \ar[dr] \ar[dl] \\
\p^2 \ar@{-->}[rr]_{b} && \p^2
}
\end{gather}
where $b$ is a Bertini involution (this link is noted $\chi_8$ in \cite{IKT}).
Recall that geometrically, $b$ is defined as follows.
Since $p$ is general, the linear system $\Gamma$ of cubics through $p$ is a pencil whose general member is smooth. 
The base locus of the pencil is equal to $p \cup q$, where $q$ is a point of degree 1. 
For $x \in \p^2$ a general point, the unique member of $\Gamma$ through $x$ is a smooth cubic, that we can see as an elliptic curve with neutral element $q$.
Then $b(x) = -x$, where $-x$ means the opposite of $x$ with respect to the group law on the elliptic curve. 

In particular, such a link $b$ is an elementary generator as defined above.
Now the crucial but easy observation is:

\begin{lemma}\label{lem:chi8}
In diagram $(\ref{diag:chi8})$, each contraction $S \to \p^2$ corresponds to an
edge in $\Xl$ which is not contained in any square.
\end{lemma}

\begin{proof}
If $S/\pt \to \p^2/\pt$ corresponds to an edge of a square, we would have a rank 3 fibration $S'/\pt$ that factorizes through $S/\pt$.
But such a surface $S'$ would be a del Pezzo surface, and would be a blow-up of $S$ which is already a del Pezzo surface of degree 1; contradiction.
\end{proof}

Up to changing the initial choice of representatives $\E$, we can assume that a representative of a Bertini involution is an involution. 
We denote by $\B \subseteq \E$ the subset of representatives of Bertini involutions.

\begin{example}
Since the above construction relies on the existence of a Galois extension of degree 8, we recall a few examples of such extensions:
\begin{enumerate}[wide]
\item
$\Q(\sqrt{2}, \sqrt{3}, \sqrt{5})/\Q$ is Galois with Galois group isomorphic to $(\Z/2\Z)^3$.
This is the splitting field of $(X^2 - 2)(X^2 -3)(X^2-5)$.

\item
The cyclotomic extension $\Q(e^{2i\pi/15})/\Q$ has degree $\phi(15) = 8$ and Galois group isomorphic to $(\Z/15\Z)^* \simeq \Z/2\Z \times \Z/4\Z$.
This is the splitting field of the 15th cyclotomic polynomial $\Phi_{15}(X) \in \Z[X]$ (\cite[Corollary 7.8]{Morandi}).

\item
$\Q(\sqrt[4]{2}, i)/\Q$ is Galois with Galois group isomorphic to the dihedral group $D_8$.
This is the splitting field of $(X^4 - 2)(X^2 + 1)$.
Generators for the Galois group are $r,s$ where 
\[r(\sqrt[4]{2}) = i\sqrt[4]{2},\; r(i) = i 
\text{ and }
s(\sqrt[4]{2}) = \sqrt[4]{2},\; s(i) = -i.
\]

\item
Let $\alpha = e^{i\pi/4}$, and set $\K = \Q(\alpha)$.
Then pick $\beta \in \K$ which is not a square in $\K$: for instance $\beta = 3$
is a possible choice, but $\beta = 2$ is not because $\alpha + \alpha^7 = \sqrt2$.
Then $\K(\sqrt[8]{\beta})/\K$ is a cyclic extension of degree 8, that is, Galois
with Galois group isomorphic to $\Z/8\Z$ (\cite[Corollary 9.6]{Morandi}).

\item
$\Q(\theta)/\Q$ with $\theta^2 = (2+\sqrt2)(2+\sqrt3)(3+\sqrt6)$ is Galois with
Galois group isomorphic to the quaternionic group.
This is the splitting field of $X^8-72X^6+ 180X^4- 144X^2+36$ (\cite{Dean}).

\item
Let $\FF{q}$ be a finite field, and $\FF{q^n}/\FF{q}$ the (essentially unique)
extension of degree $n$ (in particular one can take $n = 8$).
Then this extension is Galois, with Galois group isomorphic to $\Z/n\Z$, generated by the Frobenius automorphism $x \mapsto x^q$ (\cite[Corollary 6.7]{Morandi}).
\end{enumerate}
\end{example}

Now we turn to the problem of proving that the set $\B$ is large, that is, there exist many Bertini involutions, even modulo the action of $\PGL_3(\K)$.
We shall produce points of degree 8 in general position by using nodal cubics. 
The following set-up about the group structure on the smooth locus of a nodal cubic is classical, see for instance \cite[\S III.2, Proposition 2.5]{Silverman}, and also  Remark \ref{rem:group law}.
However since we want to work over a perfect field (typically non-algebraically closed), we find convenient to make a slightly different 
choice of normal form. 

\begin{setup} \label{setup:nodal cubic}
Consider the plane nodal cubic curve $C_P$ defined over $\K$, given by the equation
\[xyz = P(x,z),\]
where $P(x,z) = c_0 x^3 - c_0 z^3$, $c_0 \in \K^*$.
In the affine chart $z= 1$, the equation becomes
\[
y = \frac{P(x,1)}{x} \left( = \frac{c_0x^3 - c_0}{x} \right).
\]
Observe also that the singular point $[0:1:0]$ is the only intersection point
between $C_P$ and the line at infinity $z = 0$.
\end{setup}

We shall be interested in nodal cubics up to the action of $\PGL_3(\K)$, and we shall use the above set-up as a normal form.
Observe that we cannot assume $c_0 = 1$ even after applying a diagonal element of $\PGL_3(\K)$, because $c_0$ might not be a cube in $\K$.

\begin{lemma} \label{lem:general cubic}
Let $C \subset \p^2$ be an irreducible nodal cubic with singular point at $[0:1:0]$, and tangent cone at this point given by $xz = 0$.
Then $C$ admits an equation of the form
\[xyz = c_0 x^3 + c_1 x^2z + c_2xz^2 +c_3 z^3,
 \qquad c_0, c_3 \in \K^*, c_1, c_2 \in \K,\]
and $C$ is equivalent under the action of $\PGL_3(\K)$ to a cubic from Set-Up \ref{setup:nodal cubic} if and only if $-\frac{c_0}{c_3}$ is a cube in $\K$.
\end{lemma}

\begin{proof}
The assumption on the singular point implies that $C$ admits an equation of the form $xyz = P(x,z)$ with $P$ a homogeneous polynomial of degree 3.
Moreover $P(0,1) \neq 0$ and $P(1,0) \neq 0$, otherwise $C$ would be reducible.
So $P(x,z) = c_0 x^3 + c_1 x^2z + c_2xz^2 +c_3 z^3$, with $c_0, c_3 \in \K^*$, and no condition on $c_1, c_2$.

By applying $(x,y,z) \mapsto (x,y+c_1x + c_2z,z)$ we can assume $c_1 = c_2 = 0$.
Then, if $-\frac{c_0}{c_3} = a^3$ for some $a \in \K^*$, by applying $(x,y,z) \mapsto (x,a^{-1}y, az)$ we get $c_0 = -c_3$.
\end{proof}

\begin{lemma} \label{lem:produit 1}
Assume Set-up $\ref{setup:nodal cubic}$, and consider a collection of three or six pairwise distinct $a_i \in \K^*$. 
Then:
\begin{enumerate}
\item The points $p_i = \left(a_i,  \frac{P(a_i,1)}{a_i}\right) \in C_P$, $i = 1,2,3$, are on a same line if and only if $a_1 a_2 a_3 = 1$.

\item  The points $p_i = \left(a_i,  \frac{P(a_i,1)}{a_i}\right) \in C_P$, $i = 1,\dots,6$, are on a same conic if and only if $a_1 \dots a_6 = 1$.
\end{enumerate}
\end{lemma}

\begin{proof}
\begin{enumerate}[wide]
\item
The general equation of a line that does not intersect the cubic $C_P$ at infinity is $y + Ax + B = 0$.
If the $p_i$ are on a same line, then replacing $y = \frac{P(x,1)}{x}$ in the equation of the line we get
\[P(x,1) + Ax^2 + Bx   = c_0(x-a_1)(x-a_2)(x-a_3),\]
hence, by comparing the constant terms and dividing by $-c_0$, we get $1 = a_1a_2a_3$ as expected.
Conversely, if $a_1a_2a_3 = 1$, then the above relation allows to define $A, B$ in function of the $a_i$ such that the line $y + Ax + B = 0$ contains the three points $p_i$.

\item The proof is similar, working with the general equation of a conic that does not intersect the cubic at infinity:
\[y^2 + Ayx + By + Cx^2 + Dx + E=0.\]
Replacing $y = \frac{P(x,1)}{x}$ we get
\begin{multline*}
P(x,1)^2 + Ax^2P(x,1) + BxP(x,1) + Cx^4 + Dx^3 + Ex^2 \\
= c_0^2 x^6 + \dots + c_0^2=  c_0^2\prod_{i=1}^6 (x-a_i),
\end{multline*}
and we conclude as in the previous case. \qedhere
\end{enumerate}
\end{proof}

\begin{remark}\label{rem:group law}
On the smooth locus $C^{sm}$ of a nodal cubic $C$, with a choice of $e$ a flex point, 
recall that there exists a group law defined similarly as in the case of an elliptic curve.
Given $p, q \in C^{sm}$, define $p \circ q$ as the third point of intersection of $C$ with the line through $p$ and $q$ (or the tangent by $p$ if $p = q$).
Then by setting $p\cdot q := (p \circ q) \circ e$, we get a group law $\cdot$ with neutral element $e$.
The previous lemma shows that in the case of $C_P$ given by Set-Up \ref{setup:nodal cubic}, where one can check that $e = (1,0)$ is a flex point, the map $x \mapsto (x,\frac{P(x,1)}{x})$ is a group morphism from $\K^*$ to $C_P^{sm}$. 
\end{remark}

We shall need the following result about singular fibers of a pencil of cubic curves:

\begin{lemma} \label{lem:12 nodal cubics}
Let $\Gamma$ be a pencil of plane cubic curves.
Then $\Gamma$ contains at most $12$ nodal cubics.
\end{lemma}

\begin{proof}
Consider the surface $S$ obtained by blowing-up $\p^2$ at the nine base points of the pencil.
Then $S$ admits an elliptic fibration and has Euler number $c_2(S) = 12$, which by Ogg's formula is equal to the sum over the singular fibers of the valuation $v(\Delta)$ of the minimal discriminant. 
Independently of the characteristic of the base field, each nodal cubic contributes by 1 to this sum, hence the result (see e.g. \cite[\S 5.3]{liedtke}).
\end{proof}

Now we apply this set-up to the case of a field which admits a Galois extension of degree 8.

\begin{proposition} \label{pro:general position}
Assume Set-Up $\ref{setup:nodal cubic}$.
Let $\L/\K$ be a Galois field extension of degree 8, $b_1, \dots, b_8 \in \L$ be an orbit under $\Gal(\L/\K)$, and $\lambda \in \K^*$. 
Set $a_i = \lambda b_i$, so that $a_1, \dots, a_8 \in \L$ also is a Galois orbit.
Then, except for at most $6$ values of $\lambda$, the points $p_i = (a_i, \frac{P(a_i,1)}{a_i}) \subset \A^2 \subset \p^2$ are in general position.
\end{proposition}

Before giving the proof we establish a corollary.

\begin{corollary}
Let $\K$ be an infinite field that admits at least one Galois extension of degree $8$.
Then the set $\B$ of representatives of Bertini involutions up to conjugacy by $\PGL_3(\K)$ has at least the same cardinality than the field $\K$.
\end{corollary}

\begin{proof}
Assume that $a_1,\dots, a_8$ is a Galois orbit such that the points  $p_i = \big(a_i, \frac{P(a_i,1)}{a_i}\big)$ are in general position.
It is sufficient to prove that for any $\lambda \in \K^*$ except finitely many, the points $q_i =  \big(\lambda a_i, \frac{P(\lambda a_i,1)}{\lambda a_i}\big)$ also are in general position and are not equivalent to the $p_i$ under the action of $\PGL_3(\K)$.
First by Proposition \ref{pro:general position}, by avoiding 6 values of $\lambda$ we can assume that the $q_i$ are in general position.
Assume that $g \in \PGL_3(\K)$ sends the $p_i$ on the $q_i$.
By assumption the nodal cubic $C_P$ from Set-Up \ref{setup:nodal cubic} contains both the $p_i$ and the $q_i$.
Then $g^{-1}(C_P)$ is a nodal cubic through the $p_i$, and by Lemma \ref{lem:12 nodal cubics} we know that there are at most 12 of them.
Moreover a given nodal cubic is stabilized by finitely many elements of $\PGL_3(\K)$, hence the result.
\end{proof}

We separate the proof of Proposition  \ref{pro:general position} into the next three lemmas, where we always assume Set-Up \ref{setup:nodal cubic}.

\begin{lemma} \label{lem:no cubic}
There is no singular irreducible cubic passing through all the $p_i$, with one of them the double point.
\end{lemma}

\begin{proof}
If $C$ is such a cubic, and $C' = \sigma(C)$ is the image of $C$ under a non trivial element $\sigma \in \Gal(\L/\K)$, then $C\cdot C' = 2 + 2 + \underbrace{1+ \dots + 1}_{6} = 10$, a contradiction.
\end{proof}

\begin{lemma} \label{lem:no line}
Any three points among the $p_i = \big( a_i, \frac{P(a_i,1)}{a_i} \big)$ are not on a line.
\end{lemma}

\begin{proof}
Assume the contrary, and denote by $L$ a line containing three of the $p_i$.
Since the $p_i$ lie on a nodal cubic, the line $L$ contains exactly three points among the $p_i$.
Then taking the Galois orbit of $L$ we obtain a configuration of 8 lines, each line containing 3 of the $p_i$, and with 3 lines through each $p_i$.
This is the classical Möbius-Kantor configuration $8_3$, see figure \ref{fig:mobius-kantor}, where 7 of the lines are represented as lines of the plane, and the label of a vertex indicates the first coordinate of the corresponding point in $\A^2$.
Precisely, $a,b,c$ are among the $a_i$, and the other labels are expressed in terms of $a,b,c$ using Lemma \ref{lem:produit 1}.

Multiplying the labels of the bottom line we find $c^3 = 1$.
But a cubic root of the unity either lives in $\K$ or in a quadratic extension of $\K$, in contradiction with the assumption that $c$, being any one of the $a_i$, should satisfy $\K(c) = \L$.
\end{proof}

\begin{figure}[ht]
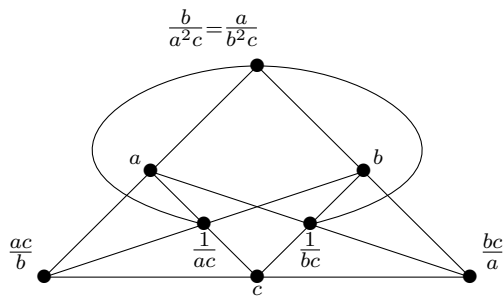

\[\mygraph{
!{<0cm,0cm>;<1.4cm,0cm>:<0cm,1.4cm>::}
!{(0,2)}*-{\rond}="0"
!{(-1,1)}*-{\rond}="1"
!{(1,1)}*-{\rond}="2"
!{(-0.5,.5)}*-{\rond}="3"
!{(.5,.5)}*-{\rond}="4"
!{(-2,0)}*-{\rond}="5"
!{(0,0)}*-{\rond}="6"
!{(2,0)}*-{\rond}="7"
"1"-"4"-_<{\quad \tfrac{1}{bc}}"7"
"1"-"3"-"6"
"1"-_>{\tfrac{ac}{b}}"5"
"2"-"3"-^<{\tfrac{1}{ac}\quad}"5"
"2"-"4"-"6"
"2"-^>{\tfrac{bc}{a}}"7"
"5"-_>c"6"-"7"
"1"-^<a^>(1.1){\tfrac{b}{a^2c} = \tfrac{a}{b^2c}}"0"-^>b"2"
"3"-@`{c+(-2,.5),c+(2.5,2),p+(2,.5)}@{-}"4"
}\]
\caption{The Möbius-Kantor configuration $8_3$} \label{fig:mobius-kantor}
\end{figure}

\begin{lemma} \label{lem:no conic}
Except for at most $6$ values of $\lambda$, any $6$ points among the points $p_i$ associated to the orbit $\lambda b_i$ are not on a conic.
\end{lemma}

\begin{proof}
Assume $C$ is a conic through $6$ of the $p_i$.
Since the $p_i$ lie on the cubic $C_P$, Bézout Theorem implies that the 2 remaining points $p_i$ are not on the conic $C$.
Consider $\sigma$ in the Galois group $\Gal(\L/\K)$.
Then $\sigma(C)$ is either equal to $C$, or shares exactly 4 intersection points with $C$.
The only possibility is that the Galois orbit of $C$ consists of 4 conics, and we can group the $p_i$ into 4 pairs, such that each conic passes through three of the four pairs of points.
Then the combinatorics is a configuration $4_3$: by each pair pass 3 conics, and each conic contains three pairs of points.
Denote by $\pi_1, \pi_2, \pi_3, \pi_4$ the product of the first coordinates of each pair of points.
By Lemma \ref{lem:produit 1} we have
\begin{align*}
\pi_1  \pi_2  \pi_3 &= 1, & \pi_1  \pi_2  \pi_4 &= 1, \\
\pi_1  \pi_3  \pi_4 &= 1, & \pi_2  \pi_3  \pi_4 &= 1.
\end{align*}
It follows that $\pi_1 = \pi_2 = \pi_3 = \pi_4$, and then $\pi_i^3 = 1$ for all $i$.
Finally pick an element $\lambda \in \K^*$, such that $\lambda^6 \neq 1$.
Now if we replace $a_i$ by $\lambda a_i$, and the corresponding points on $C_P$ are still in non-general position, then the same argument
gives 
\[\lambda^6  = \lambda^6 \pi_i^3 = (\lambda a_i \tau(\lambda a_i))^3 = 1,\]
a contradiction.
We conclude that if there is a bad point of degree 8 associated to an orbit $a_i$, then for all $\lambda \in \K^*$ except maybe the 6th roots of unity, the point of degree 8 associated with the orbit $\lambda a_i$  is in general position with respect to conics.
\end{proof}

In the case of finite fields, for small cardinal the previous statements are empty, and anyways for arbitrary cardinal one would like an estimate of the cardinality of $\B$.
We adapt the previous argument as follows.

\begin{lemma} \label{lem:finite field}
Assume Set-Up $\ref{setup:nodal cubic}$ for a finite field $\K = \FF{q}$.
Let $a_1, \dots, a_8 \in \FF{q^8}$ be an orbit under $\Gal(\FF{q^8}/\FF{q})$, and assume that the points $p_i = (a_i, \frac{P(a_i,1)}{a_i}) \in \A^2 \subset \p^2$ are in non-general position.
For $\beta \in \FF{q^4}^*$, set $b_i = \beta^{q^{i-1}} a_i$, so that $b_1, \dots, b_8 \in \FF{q^8}$ is also a Galois orbit.

If $\beta \not\in \{x \in \FF{q^4}^*;\; x^6 = 1, x^2 \in \FF{q}\}$, then the points $q_i = (b_i, \frac{P(b_i,1)}{b_i}) $ are in general position.
\end{lemma}

\begin{proof}
Lemmas \ref{lem:no cubic} and \ref{lem:no line} are valid for any $\beta$, so the delicate point is only the general position with respect to conics.
With the same notation as in the proof of Lemma \ref{lem:no conic}, since 6 among the $p_i$ lie on a conic $C$ then we obtain that the products $\pi_i$  are all equal to a same 3rd root of unity.
Now we use the fact that over a finite field, $\tau:x \mapsto x^{q^4}$ is the only element of order 2 in $\Gal(\FF{q^8}/\FF{q})$.
In particular, $\tau$ is the element of order 2 that fixes $C$. It also fixes all conjugates of $C$ and hence interchanges the two elements in each pair of $p_i$.
Then the product $\pi_i$ is invariant under the Galois group $\Gal(\FF{q^8}/\FF{q})$, so we have $\pi_i \in \FF{q}$.
Pick $\beta \in \FF{q^4} \setminus  \{x \in \FF{q^4}^*;\; x^6 = 1, x^2 \in \FF{q}\}$, and replace the orbit of $a_1$ by the orbit of $b_1 = \beta a_1$, which still has cardinal 8.
Then $\pi_1 \in \FF{q}$ is replaced by $\beta^2 \pi_1$, which either is not a 3rd root of the unity anymore, or is not an element of $\FF{q}$.
Thus for each such choice of $\beta$ the points $q_i =(b_i, \frac{P(b_i,1)}{b_i})$ are in general position with respect to conics.
\end{proof}

\begin{lemma} \label{lem:same orbit}
Let $x, x'$ be two conjugate elements in  $\FF{q^8} \setminus \FF{q^4}$, that are also in the same orbit under the action of $\FF{q^4}^*$ by left multiplication.
Assume that one of the following conditions holds:
\begin{enumerate}[$(i)$]
\item $x$ is a generator of $\FF{q^8}$;
\item $q = 2$.
\end{enumerate}
Then $x = x'$.
\end{lemma}

\begin{proof}
Assume $x \neq x'$.
We have $x' = x^{q^i}$, for some $1 \le i \le 7$, and $x'/x = x^{q^i - 1} \in \FF{q^4}^*$.
In particular $x^{(q^i - 1)(q^4 -1)} = 1$, so that $(q^i - 1)(q^4 -1)$ is a non-zero multiple of $\order(x)$.

If $x$ is a generator of $\FF{q^8}$, we get $(q^i - 1)(q^4 -1) = d(q^8 -1)$ for some $d > 0$.
This implies $5 \le i \le 7$, and reducing modulo $q^4$ we find $d \equiv -1 \mod q^4$, hence $d \ge q^4 -1$ which gives a contradiction.

If $q = 2$, the group $\FF{2^8}^*$ is cyclic of order $255 = 3\cdot5\cdot17$.
Observe that an element $x \in \FF{2^8}^*$ is in $\FF{2^8} \setminus \FF{2^4}$ if and only if $\order(x)$ is a multiple of 17 (namely, the possibilities are $17, 51, 85$ and $255$).
Then one checks that for $1 \le j \le 7$, $2^j -1$ is not a multiple of 17, which gives the expected contradiction.
\end{proof}

\begin{proposition}\label{cor:B count}
Let $\FF{q}$ be a finite field.
Then the number of Bertini involutions with a base point of degree $8$, up to conjugacy by $\PGL_3(\FF{q})$, is at least $M_q$, where $M_2 = 2$, $M_3 = 12$, and for $q \ge 4$, 
\[M_q =\frac{1}{640} (q^6-1).\]
In particular, $M_q \ge q$ for all $q \ge 2$.
\end{proposition}

\begin{proof}
First we count the number of nodal cubics equivalent to the ones from Set-Up~\ref{setup:nodal cubic}.

We need to choose a point in $\p^2({\FF{q}})$, and then two distinct lines defined over $\FF{q}$ through this point. 
The number of such choices is:
\[N_1 = (q^2 + q + 1) \frac{q(q+1)}2.\]
Using the action of $\PGL_3(\FF{q})$, we can assume that the two lines are $x = 0$ and $z = 0$, and by Lemma \ref{lem:general cubic} we need to count nodal cubics with equation of the form
\[xyz = c_0 x^3 + c_1 x^2z + c_2xz^2 +c_3 z^3,
 \qquad c_0, c_3 \in \FF{q}^*, c_1, c_2 \in \FF{q},\]
and $-c_0/c_3$ a cube in $\FF{q}$.
The number of choices is at least 
\[N_2 =  \frac{(q-1)^2q^2}{3}.\]
In fact, if 3 does not divides $q-1$, then any element of $\FF q$ is a cube, and we do not need to divide by 3 in the above formula (we shall use this remark below, for $q = 2$ or $3$).

Consider a Galois orbit $a_1, \dots, a_8$, and the associated points $p_i =\big( a_i, \frac{P(a_i,1)}{a_i}\big)$.
If the $p_i$ are in non-general position, then Lemma \ref{lem:finite field} says that by multiplying by $\beta \in \FF{q^4}^* \setminus \{x;\; x^6 = 1, x^2 \in \FF{q}\}$ we can produce Galois orbits in general position.
Moreover if $a_i$ is a generator of $\FF{q^8}^*$, by Lemma \ref{lem:same orbit} these orbits are pairwise disjoint.
The number of such orbits in general position is bounded from below by the rational $(q^4-7)/(q^4-1)$, which is greater than $9/10$ for $q \ge 3$.
In fact, for $q = 2$ the only third root of unity with square in $\FF{2}$ is $1$, so again we get the ratio $(2^4-2)/(2^4 - 1) > 9/10$.

The number of generators for $\FF{q^8}^*$ is equal to $\phi(q^8 - 1)$, where $\phi$ is the Euler function.
We have the following lower bound for the Euler function \cite[Theorem 15]{RS}:
\[\phi(n) \ge \frac{n}{e^\gamma \log(\log n) + \frac{3}{\log(\log n)}},\]
where $\gamma$ is the Euler constant.
One can check that for $q \ge 3$, this implies $\phi(q^8-1) \ge q^6 - 1$.
So we get at least
\[N_3 = \frac{9}{10}\frac{q^6-1}{8}\]
Galois orbits of cardinal 8 in general position on a given nodal cubic.
Finally by Lemma \ref{lem:12 nodal cubics}, a given orbit belongs to at most 12 nodal cubics, and we also have to mod out by the action of $\PGL_3(\FF{q})$ whose cardinal is:
\[|\PGL_3(\FF{q})| = \frac{1}{q-1}(q^3 - 1)(q^3 - q)(q^3-q^2) = q^3 (q^3-1)(q^2-1).\]
Finally:
\begin{align*}
M_q &= \frac{N_1 \cdot N_2 \cdot N_3}{12 |\PGL_3(\FF{q})|} \\
    &= \frac{9}{2\cdot3\cdot80\cdot12} \frac{(q^2 + q + 1) q(q+1) \cdot  (q-1)^2q^2 \cdot
(q^6-1)}
    {q^3 (q^3-1)(q^2-1)} \\
   &= \frac{1}{640} (q^6-1).
\end{align*}

For $q = 3$, we do not need the 3 in the denominator of $N_2$, and we can replace the coarse estimate $3^6-1$ in the formula of $N_3$
by the exact number $\phi(3^8-1) = 2560$, so that we get the better bound:
\begin{align*}
M_3 &= \frac{9 \cdot 2560}{2 \cdot 80 \cdot 12}\frac{(3^2 + 3 + 1) 3(3+1) \cdot 
 (3-1)^23^2}{3^3
(3^3-1)(3^2-1)} \\
&= 12.
\end{align*}

Finally for $q=2$, by Lemma \ref{lem:same orbit} we can use all 240 elements of $\FF{2^8} \setminus \FF{2^4}$ in the estimate for $N_3$, and not only the generators of $\FF{2^8}^*$. 
Moreover as above we can discard the 3 in the denominator of $N_2$, and we find:
\[N_1 = 21, \; N_2 = 4, \; N_3 = 27,\; |\PGL_3(\FF2)| = 168,\]
and
\begin{align*}
M_2 &= \frac{21 \cdot 4 \cdot 27}{12 \cdot 168} = \frac98 > 1,\\
\end{align*}
which we can round-up to $M_2 = 2$.
\end{proof}

\subsection{Jonquières maps}\label{subsec:jonq}

Let $\K$ be any field, $n \ge 2$ any dimension.
We define the \textit{Jonquières subgroup} $J \subset \Bir_\K(\p^n)$ as the subgroup isomorphic to $\PGL_2(\K(x_2, \dots, x_n))$, via the choice of an affine chart $\A^n \subset \p^n$ and the formula 
\[(x_1, \dots, x_n) \rat \left( \frac{A(x_2, \dots, x_n)x_1 + B(x_2, \dots, x_n)} {C(x_2, \dots, x_n)x_1 + D(x_2, \dots, x_n)} , x_2, \dots, x_n \right).\]
Recall that a group $G$ is called \textit{perfect} if it is equal to its commutator subgroup $G^{(1)}$.
We define the special Jonquières subgroup as the commutator subgroup of $J$; this is a group  isomorphic to $\PSL_2(\K(x_2, \dots, x_n))$.
Let $G$ be the subgroup of $\Bir_\K(\p^n)$ generated by the special Jonquières subgroup $J^{(1)}$ and $\Aut_{\K}(\p^n)^{(1)} = \PSL_{n+1}(\K)$. 

\begin{proposition} \phantomsection \label{pro:special jonq}
\begin{enumerate}[wide]
\item If $g \neq \id$ is an element in the special Jonquières subgroup or in $\PSL_n(\K)$, then the normal subgroup $\llangle g \rrangle$ generated by $g$ in $G$ is equal to $G$;
\item $G$ is a perfect group.
\end{enumerate}
\end{proposition}

\begin{proof}
\begin{enumerate}[wide]
\item The groups $\PSL_2(\K(x_2, \dots, x_n))$ and $\PSL_{n+1}(\K)$ are simple (recall that we assume $n \ge 2$, so we avoid the non simple groups $\PSL_2(\FF 2)$ and $\PSL_2(\FF 3)$), and they have a non trivial intersection, as they both contain for instance the translation $(x_1 + 1, x_2, \dots, x_n)$.
From these facts it follows  that $\llangle g \rrangle$ contains both $\PSL_2(\K(x_2, \dots, x_n))$ and $\PSL_{n+1}(\K)$, hence $G$.

\item Follows from the fact that $\PSL_2(\K(x_2, \dots, x_n))$ and $\PSL_{n+1}(\K)$ are both perfect groups (because simple and non-abelian).
\qedhere
\end{enumerate}
\end{proof}

\begin{corollary} \phantomsection \label{cor:jonquieres}
\begin{enumerate}[wide]
\item If $\phi\colon \Bir_\K(\p^n) \to A$ is a morphism to an abelian group $A$, then $G = \langle \PSL_2(\K(x_2, \dots, x_n)) , \PSL_{n+1}(\K) \rangle \subset \ker \phi.$
\item If a morphism $\phi\colon\Bir_\K(\p^n)\rightarrow H$ sends a non-trivial element $g\in G$ onto $1_H$, then $G\subset\ker\phi$.
\end{enumerate}
\end{corollary}

Now we come back to the case of dimension 2.
An equivalent definition of Jonquières map is that $f \in  \Bir (\p^2_\K)$ is Jonquières if it admits a base point $q$ of degree 1 and $f$ preserves a general member of the pencil of lines through $q$.
For instance a quadratic map with one base point of degree 1 and one base point of degree 2 is Jonquières, but a quadratic map with a unique base point of degree 3 is not Jonquières.
With the identification $(x,y) \in \A^2 \mapsto [x:y:1] \in \p^2$, a Jonquières map is written $(x,y) \mapsrat \left( \frac{A(y)x + B(y)}{C(y)x + D(y)}, y\right)$ and admits the degree 1 point $q = [1:0:0]$ as a base point.
Over a perfect field $\K$ one can factorize such a map into Sarkisov links by first blowing-up $q$ to get a surface $\F_1$, then performing a sequence of type II links over $\p^1$ between Hirzebruch surfaces, and a last contraction to come back to $\p^2$.
In particular, any Jonquières map $f$ is an elementary generator, so that $\E$ contains a representative equivalent to $f$.

Recall that by Noether-Castelnuovo Theorem, over an algebraically closed field $\K$  the Jonquières group $J$ and the automorphism group  $\PGL_{3}(\K)$ generate the Cremona group $\Bir_\K(\p^2)$.
In this context Corollary \ref{cor:jonquieres} implies that $J$ and $\PGL_{3}(\K)$ embed into any non-trivial quotient of $\Bir_\K(\p^2)$, in particular such a quotient cannot be finite, nor abelian.
Observe also that there exists some non-algebraically closed field $\K$ such that $\Bir (\p^2_\K) = \langle J, \PGL_3(\K) \rangle$, so that the same remark applies: by \cite{Isko}, it is sufficient that $\K$ does not admit any extension of degree $\le 8$.
For instance, one can take $\K$ to be the composite (which in this case is also the union) of all towers of extensions of $\Q$ of degree at most 8. 

Finally, remark that a Jonquières map can have base points of arbitrary degree, and in particular of degree 8.
For instance, over a suitable field $\K$ let $p,q \in \p^2$ be points of respective degree 8 and 1.
Then blowing-up the point $q$, and performing an elementary link from $\F_1$ to $\F_1$ by blowing up $p$ and contracting the orbit of 8 fibers through $p$, we get a Jonquières map $j \in \Bir_\K(\p^2)$, that we can choose to be an involution (up to composing by an automorphism of $\p^2$).
On the other hand, by blowing up only $p$, we construct a Bertini involution which has nothing to do with $j$.
In fact it follows from Theorem \ref{thm:main big amalgam} that $b$ and $j$ generate an infinite dihedral group $\Z/2\Z * \Z/2\Z$.

\section{Amalgamated structure and morphism to a free product}
\label{sec:proofs}

\subsection{Bass-Serre tree of an amalgam}

Our reference for this section is \cite{Serre_Trees}.
Let $G$ be a group, $A$ a subgroup, and $(G_i)_{i \in I}$ a collection of (proper) subgroups generating $G$ and such that $A \subset G_i$ 
for all $i$.

One constructs a graph $\Gl$ on which $G$ acts as follows.
The vertices are left cosets $g A$ and $g G_i$ in $G/A$ and $G/G_i$.
Then for each $g \in G$ and each $i \in I$, we put an edge between the vertices $g A$ and $g G_i$.
The group $G$ acts on the resulting graph $\Gl$ by 
\begin{align*}
f \cdot g A &:= (fg)A, &  f \cdot g G_i &:= (fg)G_i.
\end{align*}

One says that $G$ is the \textit{amalgamated product} of the $G_i$ over $A$, denoted $G = \bigast_A G_i$, if it satisfies the following
universal property: for any group $H$, and any collection of morphisms $\phi_i\colon G_i \to H$ that coincide on $A$, there exists a (necessarily unique) morphism $\phi\colon G \to H$ that extends all the $\phi_i$.
In this case, one can show that $A = G_i \cap G_j$ for all $i \neq j$ (see \cite[Remark after Theorem 1 p.3]{Serre_Trees}).

Recall that a \textit{star graph} is a tree of diameter 2. 
We call \textit{central vertex} the unique center of a star graph, and \textit{peripheral vertices} the other vertices.
When the group $G$ is the amalgamated product of the $G_i$ over $A$ then the graph $\Gl$ is a tree, and a fundamental domain with respect to the action of $G$ is the star graph pictured on Figure \ref{fig:treeFD}, where we label each vertex by its stabilizer.
Conversely we have the following basic result from Bass-Serre Theory:

\begin{theorem}[{\cite[\S4, Theorem 10]{Serre_Trees}}]\label{thm:bass-serre}
With the notation above, suppose that $G$ acts on a tree $\Tl$ with fundamental domain as in Figure~$\ref{fig:treeFD}$, such that
\begin{itemize}
\item $A$ is the stabilizer of the central vertex in the fundamental domain
\item the $G_i$ are the stabilizers of the peripheral vertices in the fundamental domain.
\end{itemize}
Then $G = \bigast_A G_i$ is the amalgamated product of the $G_i$ along $A$, and the graph $\Gl$ constructed above is isometric to $\Tl$.
\end{theorem}

\begin{figure}[ht]
\def\svgwidth{.3\textwidth}
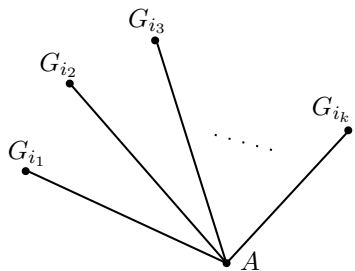
\caption{Fundamental domain for the Bass-Serre tree of the
amalgam $\bigast_AG_i$.}\label{fig:treeFD}
\end{figure}

\begin{remark} \label{rem:2 factors}
In the case where there are only two subgroups $G_1$ and $G_2$ with intersection $A = G_1 \cap G_2$, it is customary to remove the vertices of valence 2 corresponding to the left cosets modulo $A$, and to consider that the fundamental domain is a single edge with the following stabilizers:
\[\mygraph{
!{<0cm,0cm>;<1cm,0cm>:<0cm,1cm>::}
!{(-1,0)}*{\bullet}="G1"
!{(1,0)}*{\bullet}="G2"
"G1"-^<{G_1}^A^>{G_2}"G2"
}\]
\end{remark}

\subsection{Subcomplexes}

First we define two subcomplexes of the complex $\Xl$ constructed in \S\ref{subsec:full complex}.
Recall that we say that a vertex in $\Xl$ has type $\p^2$ if it is of the form $(\p^2, \phi)$ for some $\phi \in \Bir_\K(\p^2)$.

Let $\Xl_\B\subset\Xl$ be the subgraph whose edges correspond to blow-ups of degree 8 points in $\p^2$.
If $b \in \B$, we denote $\eta_b\colon S_b \to \p^2$ the blow-up of the base point of degree 8 of $b$.
In particular, for any $\phi \in \Bir_\K(\p^2)$ and any $\alpha \in \Aut_\K(\p^2)$, we obtain the following two edges in $\Xl_\B$:
\[\xymatrix@R-15pt{
& (S_b,\phi \alpha \eta_b) \ar[dl]_{\eta_b} \ar[dr] \\
(\p^2, \phi) = (\p^2, \phi \alpha)&& (\p^2, \phi \alpha b) = (\p^2, \phi \alpha
b \alpha^{-1})
}\]
Conversely, any edge in $\Xl_\B$ has this form, and any two vertices of type $\p^2$ at distance 2 in $\Xl_\B$ differ by a Bertini involution $\alpha b \alpha^{-1}$, for some $b \in \B$ and $\alpha \in \Aut_\K(\p^2)$.

We define another subcomplex $\Xl_e\subset\Xl$, by taking the closure of the complement of $\Xl_\B$ in $\Xl$.

\begin{lemma} \label{lem:intersection}
The intersection of the two subcomplexes $\Xl_\B$ and $\Xl_e$ is exactly the set of vertices of type $\p^2$:
\[\Xl_\B \cap \Xl_e = \{(\p^2, \phi) \mid \phi \in \Bir(\p^2) \}.\]
\end{lemma}

\begin{proof}
Lemma~\ref{lem:chi8} states that an edge in $\Xl$ corresponding to the blow-up of a point in $\p^2$ of degree $8$ is not contained in any square. 
Therefore,  $\Xl_\B \cap \Xl_e$ contains only vertices.
Now as observed before, there are two types of vertices in $\Xl_\B$.
A vertex of the form $(S_b,\phi \alpha \eta_b)$ belongs to exactly two edges of $\Xl$, which by definition are edges of the graph $\Xl_\B$, therefore such a vertex does not belong to $\Xl_e$.
On the other hand, any vertex of type $\p^2$ belongs to edges from both $\Xl_\B$ and $\Xl_e$, associated to blow-up of points of respective degree 8 or distinct from 8.
\end{proof}

We denote respectively $\Xl_\B^\circ\subset\Xl_\B$ and $\Xl_e^\circ \subset \Xl_e$ the connected components containing $(\p^2,\id)$.

\begin{lemma}\label{lem:Xe_and_Xb} 
\text{ }
\begin{enumerate}
\item\label{Xe} Both $\Xl_e$ and $\Xl_\B$ are invariant under the action of $\Bir_\K(\p^2)$.
\item\label{Xb} The graph $\Xl_\B^\circ$ is a tree.
\end{enumerate}
\end{lemma}

\begin{proof}
\begin{enumerate}[wide]
\item An edge in $\Xl_\B$ has the form $(S, \phi\eta) \to (\p^2,\phi)$ for some $\phi \in \Bir_\K(\p^2)$ and $\eta$ a blow-up of a point of degree 8.
Now $g \in \Bir_\K(\p^2)$ sends this edge to $(S, g\phi\eta) \to (\p^2,g\phi)$, which is again an edge of the same form, hence in $\Xl_\B$.
This gives that $\Xl_\B$ is invariant under the action of $\Bir_\K(\p^2)$, thus the same is true for the closure of its complement.

\item Assume that the graph $\Xl_\B^\circ$ is not a tree.
Then there exists a sequence of edges $e_1, \dots, e_r$ in $\Xl_\B^\circ$ that form an embedded loop.
Recall that by Proposition \ref{pro:simply connected} the complex $\Xl$ is simply connected.
By collapsing in $\Xl$ all edges of this loop except $e_1$, we obtain a space which is still simply connected, and which is the connected sum of a circle (corresponding to $e_1$) and another space.
By van Kampen theorem such a space should have fundamental group of the form $\Z * G$ for some $G$, contradiction. \qedhere
\end{enumerate}
\end{proof}

We recall the following definitions of subgroups of $\Bir_\K(\p^2)$ that were given in the introduction:
\[G_{e}:=\langle\Aut_\K(\p^2),\E\setminus\B \rangle\qquad
G_b:=\langle\Aut_\K(\p^2), b\rangle,\ b\in\B,\qquad
G_\B:=\langle\Aut_\K(\p^2),\B\rangle.\]
Observe that $G_\B=\langle G_b\mid b\in\B\rangle$.

\begin{lemma} \label{lem:stab X^o}
Let  $g \in \Bir_\K(\p^2)$.
Then
\begin{enumerate}
\item \label{stab:X^o_e} $g(\Xl_e^\circ) = \Xl_e^\circ$ if and only if $g \in G_e$.
\item \label{stab:X^o_B} $g(\Xl_\B^\circ) = \Xl_\B^\circ$ if and only if $g \in G_\B$.
\end{enumerate}
\end{lemma}

\begin{proof}
\begin{enumerate}[wide]
\item
Let $g\in\Bir_\K(\p^2)$ such that $g(\Xl_e^\circ) = \Xl_e^\circ$, and in particular $g \cdot (\p^2,\id)=(\p^2,g)\in\Xl_e^\circ$.
Pick a path $\gamma$ of edges connecting $(\p^2,g)$ and $(\p^2,\id)$ inside $\Xl_e^\circ$.
By Lemma \ref{lem:intersection}, the path $\gamma$ does not involve any edge from $\Xl_\B$.
By cutting at each intermediate vertex of type $\p^2$, we write $\gamma$ as a composition of paths $\gamma_i$, where each $\gamma_i$ links two vertices of type $\p^2$ whose markings differ by an element of $\E \setminus \B$.
It follows that $g\in G_e$.

Conversely if $g\in G_e$ we write $g=g_1\cdots g_n$ for some $g_1,\dots,g_n\in \Aut_\K(\p^2)\cup\E\setminus\B$.
Then for each $i$, there exists a path from $(\p^2,g_1\cdots g_i)$ to $(\p^2,g_1\cdots g_{i+1})$ inside $\Xl_e^\circ$.
By joining them we obtain a path in $\Xl_e^\circ$ starting at $(\p^2,\id)$ and ending at $(\p^2,g)=g\cdot (\p^2,\id)$, so that the connected component $g(\Xl_e^\circ)$ coincides with  $\Xl_e^\circ$.

\item The proof is entirely similar, and left to the reader. \qedhere
\end{enumerate}
\end{proof}

For each $b \in \B$, we define $\Tl_b$ to be the subgraph of $\Xl_\B^\circ$ obtained as the orbit of the edge between $(\p^2, \id)$ and $(S_b,\eta_b)$, under the action of $G_b = \langle\Aut_\K(\p^2), b\rangle$. 
Since $b$ stabilizes $(S_b,\eta_b)$ and $\Aut_\K(\p^2)$ stabilizes $(\p^2, \id)$, we obtain that $\Tl_b$ is connected, that is, $\Tl_b$ is a subtree of $\Xl_\B^\circ$ (which is a tree by Lemma~ \ref{lem:Xe_and_Xb}\ref{Xb}).

\begin{lemma} \label{lem:stab Tb}
Let $b, b'$ be two elements in $\B$, and $g, g' \in \Bir_\K(\p^2)$.
If $g(\Tl_b)$ and $g'(\Tl_{b'})$ are distinct, then the intersection $g(\Tl_b) \cap g'(\Tl_{b'})$ is either empty or equal to a single vertex of type $\p^2$.
\end{lemma}

\begin{proof}
By using the action of $\Bir_\K(\p^2)$, we can assume that $g' = \id$.
If $g \not\in G_\B$, then by Lemma \ref{lem:stab X^o}\ref{stab:X^o_B} $g(\Tl_b)$ and $\Tl_{b'}$ are in distinct connected components of $\Xl_\B$, and so are disjoint.
Now we assume $g \in G_\B$, so that $g(\Tl_b)$ and $\Tl_{b'}$ are two subtrees of $\Xl_\B^\circ$.

$\bullet$ First we consider the case $b \neq b'$.

Suppose that $g(\Tl_b)$ and $\Tl_{b'}$ contain a common edge.
Then in particular they contain a vertex of the form $(S_b,\phi\eta_b)$, and the two edges from this vertex.
But then the markings of the two corresponding vertices of type $\p^2$ should differ by composition by an element of the form $\alpha b \alpha^{-1} = \alpha' b' \alpha'^{-1}$ for some $\alpha, \alpha' \in \Aut_\K(\p^2)$, and this contradicts our assumption that $b, b'$ are two distinct representatives of Bertini involutions.
Now, since $g(\Tl_b)$, $\Tl_{b'}$ are two subtrees of $\Xl_\B^\circ$ without a common edge, they cannot have more than one common vertex, which has to be of type $\p^2$. 

$\bullet$ Now we assume $b = b'$.

By definition of $\Tl_b$, we have $\Tl_b =  g(\Tl_{b})$ if and only if $g \in G_b$. 
So we can assume that $g \in G_\B \setminus G_b$.
Now if $g(\Tl_b) \cap \Tl_{b}$ contains a vertex of the form $(S_b,\phi\eta_b)$, then it also contains the two neighbor vertices of type
$\p^2$. But if $(\p^2, \phi) \in g(\Tl_b) \cap \Tl_{b}$, then we should have $\phi = g f_1 = f_2$ with $f_1, f_2 \in G_b$, in contradiction with $g \not\in G_b$.
So we conclude that $g(\Tl_b) \cap \Tl_{b}$ is empty.
\end{proof}

\subsection{Quotients}

Let $\Yl$ be any connected subcomplex of $\Xl$.
We define a star graph $\Star(\Yl)$ associated with $\Yl$, by requiring that the peripheral vertices of $\Star(\Yl)$ are in one-to-one correspondence with the vertices of $\Yl$ of type $\p^2$.
Then we have a uniquely defined simplicial map from $\Yl$ to $\Star(\Yl)$, which is a bijection in restriction to the vertices of type $\p^2$, and which sends any other vertex to the central vertex of $\Star(\Yl)$.
We call $\Star(\Yl)$ together with the map $\Yl \to \Star(\Yl)$ the \textit{star quotient} of $\Yl$. 

Now assume that $(\Yl_i)$ is a collection of connected subcomplexes of $\Xl$, such that $\Xl = \bigcup_i \Yl_i$, and for any $i,j$ either $\Yl_i = \Yl_j$ or $\Yl_i \cap \Yl_j$ contains only vertices of type $\p^2$.
Then we can put together all star quotients $\Yl_i \to \Star(\Yl_i)$ in order to get a map from $\Xl$ to a well defined connected graph.

Now we check that we can apply this construction to the family of subcomplexes
\[(\Yl_i)_i = \{ g(\Xl_e^\circ);\; g \in \Bir_\K(\p^2)\} \cup \{ g(\Tl_b);\; g
\in \Bir_\K(\p^2), b \in \B\}.\]
First, 
the $g(\Tl_b)$, $g \in G_\B$, $b\in \B$, form a cover of $\Xl_\B^\circ$:
any vertex in $\Xl_\B^\circ$ at distance 2 from $(\p^2, \id)$ has the form $(\p^2, \alpha b)$ for some $b \in \B$ and $\alpha \in \Aut_\K(\p^2)$, and any edge of $\Xl_\B^\circ$ can be mapped to an edge issued from $(\p^2, \id)$ by applying an element of $G_\B$.
Lemma~\ref{lem:Xe_and_Xb}\ref{Xe} implies that the $g(\Xl_e^\circ)$ are exactly the connected components of $\Xl_e$, and in particular they are pairwise equal or disjoint. 
The same lemma implies that the $g(\Tl_b)$ cover $\Xl_\B$.
This implies that the family of complexes $(\Yl_i)_i$ is a cover of $\Xl$.
Moreover, by Lemma \ref{lem:intersection}, any intersection $g(\Xl_e^\circ) \cap g'(\Tl_b)$ contains only vertices of type $\p^2$.
Finally, by Lemma \ref{lem:stab Tb}, any intersection $g(\Tl_b) \cap g'(\Tl_{b'})$ between distinct subcomplexes is either empty or is equal to a single vertex of type $\p^2$. 

We denote by $\Tl_Q$ the resulting quotient graph, and $\sigma\colon \Xl \to \Tl_Q$ the associated simplicial map.
By a slight abuse of notation we shall use the same notation $(\p^2, \phi)$ either for a vertex of type $\p^2$ in $\Xl$, or for the corresponding vertex in $\Tl_Q$.

\begin{figure}[ht]
\def\svgwidth{.9\textwidth}
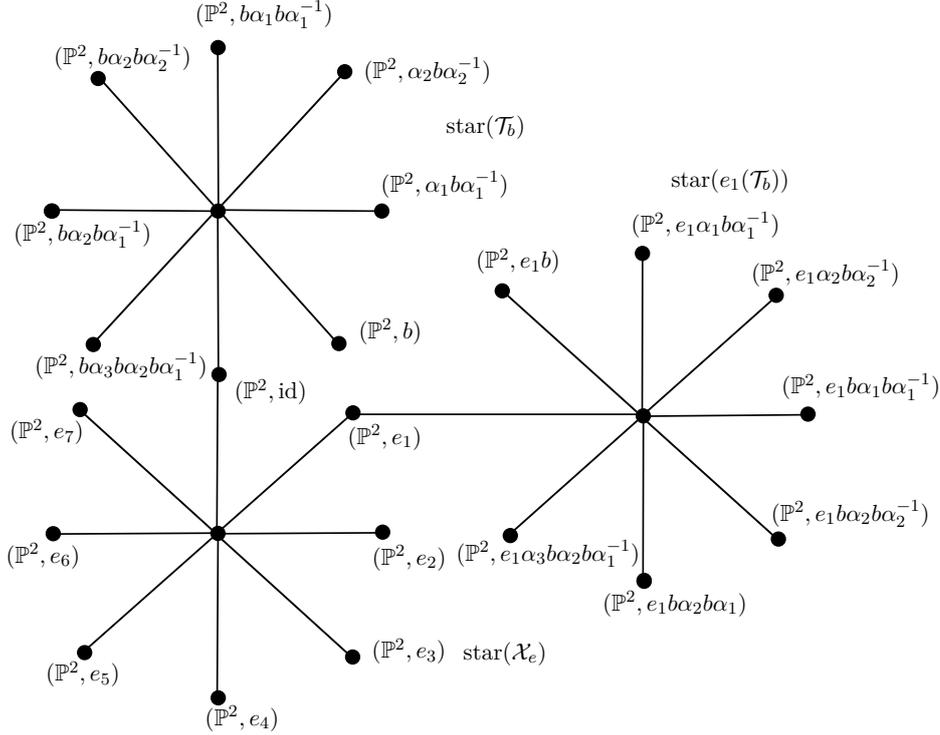
\caption{A few vertices of the tree $\Tl_Q$, where $e_i\in G_e$, $b\in\B$ and $\alpha_i\in\Aut_\K(\p^2)$.
}\label{fig:treeTQ}
\end{figure}

\begin{lemma}\label{lem:TQ}
The connected graph $\Tl_Q$ is a tree.
\end{lemma}

\begin{proof}
Let $\gamma$ be a loop in $\Tl_Q$. 
We can assume that $\gamma$ is parametrized by arc length, with base point a vertex of the form $(\p^2,\phi)$.
For each even $i$, $[\gamma(i), \gamma(i+2)]$ is a segment connecting two peripheral vertices in  $\Star(\Yl)$ for some connected subcomplex $\Yl \subset \Xl$. 
In particular, we can lift this segment as a path in $\Yl$.
Thus we obtain a lift $\tilde \gamma$ of the entire path $\gamma$, and this lift is also a closed loop because the map $\pi$ is a bijection in restriction to vertices of the form $(\p^2, \phi)$.
Now by Proposition~\ref{pro:simply connected} the loop $\tilde \gamma$ is trivial in $\pi_1(\Xl)$, hence the push-forward $\sigma_*(\tilde \gamma) = \gamma$ is trivial in $\pi_1(\Tl_Q)$.
\end{proof}

\begin{lemma}\label{lem:stabiliser}
The tree $\Tl_Q$ inherits the action of $\Bir_\K(\p^2)$ on $\Xl$, and 
\begin{enumerate}
\item \label{stabiliser:P2} the group $\Aut_\K(\p^2)$ is the stabilizer of $(\p^2,\id)$,
\item \label{stabiliser:e} the group $G_e$ is the stabilizer of the central vertex of $\Star(\Xl_e^\circ)$,
\item \label{stabiliser:b} for each $b \in \B$, the group $G_b$ is the stabilizer of the central vertex of $\Star(\Tl_b)$.
\end{enumerate}
\end{lemma}

\begin{proof}
By construction, the family
\[(\Yl_i)_i = \{ g(\Xl_e^\circ);\; g \in \Bir_\K(\p^2)\} \cup \{ g(\Tl_b);\; g
\in \Bir_\K(\p^2), b \in \B\}\]
of subcomplexes involved in the construction of  $\Tl_Q$  is invariant under the action of $\Bir_\K(\p^2)$, hence the action descends to $\Tl_Q$.
Since the quotient map $ \sigma\colon\Xl \to \Tl_Q$ is a bijection in restriction to vertices of type $\p^2$, the stabilizers of these vertices remain the same in $\Tl_Q$: this gives \ref{stabiliser:P2}.
The stabilizer of the central vertex of $\Star(\Yl_i)$ corresponds to the stabilizer of the subcomplex $\Yl_i$ for the initial action on $\Xl$.
So the two remaining assertions follow from Lemma \ref{lem:stab X^o}\ref{stab:X^o_e} and the definition of $\Tl_b$.
\end{proof}

Observe that each vertex of $\Tl_Q$ is either of the form $(\p^2, \phi)$, or is the central vertex of $\Star(\Yl)$ for some subcomplex $\Yl$ in the family.
These two types of vertices are preserved by the action of $ \Bir_\K(\p^2)$, so that $\Tl_Q$ has a natural structure of a bicolored tree.

\subsection{Structure of the Cremona group}

In this last section we prove the results stated in the introduction.

\begin{proof}[Proof of Theorem $\ref{thm:main big amalgam}$]
Since $\Tl_Q$ is a bicolored tree, and the action of $\Bir_\K(\p^2)$ on vertices of type $\p^2$ is transitive, we can look for a fundamental domain of the action inside the ball of center $(\p^2, \id)$ and radius 1. 
In fact the whole ball is a fundamental domain, indeed $\Aut_\K(\p^2)$ is the stabilizer of $(\p^2,\id)$ (see Lemma \ref{lem:stabiliser}), and we now check that $\Aut_\K(\p^2)$ also fixes each neighbor vertex.
First,  for each $b \in \B$ the central vertex of $\Star(\Tl_b)$ is a neighbor of $(\p^2, \id)$, and by Lemma \ref{lem:stabiliser} the corresponding stabilizer is $G_b$.
Then the last remaining neighbor vertex is the central vertex of $\Star(\Xl_e^\circ)$, whose stabilizer is $G_e$, by the same lemma.

By applying Theorem~\ref{thm:bass-serre} to the action of $\Bir_\K(\p^2)$ on the tree $\Tl_Q$,
we get that $\Bir_\K(\p^2)$ is isomorphic to $\bigast_{\Aut_\K(\p^2)}G_i$ where $I=\B\cup\{e\}$.

Now we prove that the action of $\Bir_\K(\p^2)$ on $\Tl_Q$ is faithful, by proving that the intersection of the stabilizers of $(\p^2, \id)$ and $(\p^2,b)$ is trivial, for any $b \in \B$.
If $g \in \Stab(\p^2, \id) \cap \Stab(\p^2, b)$, we have $g \in \Aut_\K(\p^2)$ and $bgb = g'  \in \Aut_\K(\p^2)$, so that $bg = g'b$.
But these two maps cannot have the same base points unless $g = \id$, because any automorphism of $\p^2$
preserving $8$ points in general position is the identity.
\end{proof}

\begin{remark}\label{rmk:B vide}
If the field $\K$ does not have any Galois extension of degree $8$, i.e. if $\B$ is empty, we have $\Xl_\B = \emptyset$, $\Xl_e^\circ = \Xl_e = \Xl$ and $\Tl_Q= \Star(\Xl_e^\circ)$. 
This reflects the fact that in this case $\Bir_\K(\p^2)=G_e$. 
In fact, trivially we have $\Bir_\K(\p^2)\simeq \Aut_\K(\p^2)\ast_{\Aut_\K(\p^2)\cap G_e} G_e=G_e$, and $\Tl_Q$ is its Bass-Serre tree, whose fundamental domain is the edge between $(\p^2,\id)$ and the central vertex of $\Star(\Xl_e^\circ)$.
\end{remark}

\begin{proof}[Proof of Corollary \ref{cor:main 2 amalgam}]
Let $I=\B\cup\{e\}$. Then Theorem \ref{thm:main big amalgam} gives
\[\Bir_\K(\p^2)=\bigast_{\Aut_\K(\p^2)}G_i=\left(\bigast_{\Aut_\K(\p^2)}
G_b\right)\ast_{\Aut_\K(\p^2)}G_e.\]
Now we have 
\[\bigast_{\Aut_\K(\p^2)}G_b 
= \langle \Aut_\K(\p^2), G_b\mid b \in \B \rangle
= G_\B, \]
from which the claim follows. 
The reason why $\Bir_\K(\p^2)$ acts faithfully on the Bass-Serre tree of $G_\B\ast_{\Aut_\K(\p^2)}G_e$ is the same as in the proof of Theorem~\ref{thm:main big amalgam}.
\end{proof}

\begin{proof}[Proof of Theorem \ref{thm:main morphisms}]
\begin{enumerate}[wide]
\item
Let $b \in \B$ be a Bertini involution, and consider the edge in the tree $\Xl_\B^\circ$ between the vertices $(\p^2,\id)$ and $(S_b,\eta_b)$, where $\eta_b\colon S_b \to \p^2$ is the blow-up of the base point of degree 8 of $b$.
Recall from Lemma \ref{lem:stab X^o}\ref{stab:X^o_B} that the group $G_\B$ acts on $\Xl_\B^\circ$.
The involution $b$ fixes the vertex $(S_b,\eta_b)$ and exchanges the two edges attached to it.
In particular $b$ does not fix the vertex $(\p^2,\id)$.
On the other hand, the subgroup $\Aut_\K(\p^2)$ stabilizes $(\p^2,\id)$, and any $\alpha \in \Aut_\K(\p^2)$ that also stabilizes $(S_b,\eta_b)$ must be the identity, because any automorphism of $\p^2$ preserving $8$ points in general position is the identity.
By definition the tree $\Tl_b$ is the orbit of the edge from $(\p^2,\id)$ to $(S_b,\eta_b)$, under the action of the subgroup  $G_b = \langle \Aut_\K(\p^2), b\rangle$.
Therefore the group $G_b$ acts on the tree $\Tl_b$ with fundamental domain a single edge, with stabilizer of vertices $\Aut_\K(\p^2)$ and $\langle b \rangle$, and trivial stabilizer for the entire edge.
By Theorem~\ref{thm:bass-serre} (or more precisely by using the special convention for two subgroups, see Remark \ref{rem:2 factors}),
it follows that $G_b$ is the free product
\[G_b\simeq\Aut_\K(\p^2)\ast\langle b\rangle,\]
and $\Tl_b$ is the associated Bass-Serre tree.
 
There are natural injections of $G_e$ and $G_b \simeq\Aut_\K(\p^2)\ast\langle b\rangle$ in both groups 
\begin{align*}
\Bir(\p^2) \simeq \bigast_{\Aut_\K(\p^2)} G_i &&
\text{and} &&
G_e \ast \left(\bigast_{b \in \B} \langle b \rangle \right).
\end{align*}
By the universal property of the free (or amalgamated) product we get morphisms in both direction, which are inverse to each other, and thus give the expected isomorphism
\[
\Bir(\p^2) \simeq G_e \ast \left(\bigast_{b \in \B} \Z/2\Z \right).
\]

\item
By the universal property of the free product, applied to the trivial morphism from $G_e$ and the identity map on the second factor, we get a surjective homomorphism $\Bir_\K(\p^2) \to \bigast_{\B}\Z/2\Z$.

\item The result is immediate by composing the above morphism with the abelianization morphism
\begin{equation*}
\bigast_{\B}\Z/2\Z \to \bigoplus_{\B}  \Z/2\Z. \qedhere
\end{equation*}
\end{enumerate}
\end{proof}

\begin{remark}
Another way to express point (2) of Theorem \ref{thm:main morphisms} is that we have  isomorphisms
\[\Bir_\K(\p^2)/\langle\langle G_e \rangle\rangle
\stackrel{\sim}\longrightarrow\bigast_{b\in\B}
\left(G_b/ \langle\langle\Aut_\K(\p^2)\rangle\rangle \right )\stackrel{\sim}
\longrightarrow\bigast_{b\in\B}\Z/2\Z.\]
\end{remark}

\bibliographystyle{myalpha}
\bibliography{biblio}

\end{document}

%% file: example_1.pdf_tex
\begingroup%
  \makeatletter%
  \providecommand\color[2][]{%
    \errmessage{(Inkscape) Color is used for the text in Inkscape, but the package 'color.sty' is not loaded}%
    \renewcommand\color[2][]{}%
  }%
  \providecommand\transparent[1]{%
    \errmessage{(Inkscape) Transparency is used (non-zero) for the text in Inkscape, but the package 'transparent.sty' is not loaded}%
    \renewcommand\transparent[1]{}%
  }%
  \providecommand\rotatebox[2]{#2}%
  \ifx\svgwidth\undefined%
    \setlength{\unitlength}{1520.80591397bp}%
    \ifx\svgscale\undefined%
      \relax%
    \else%
      \setlength{\unitlength}{\unitlength * \real{\svgscale}}%
    \fi%
  \else%
    \setlength{\unitlength}{\svgwidth}%
  \fi%
  \global\let\svgwidth\undefined%
  \global\let\svgscale\undefined%
  \makeatother%
  \begin{picture}(1,0.26857087)%
    \put(0,0){\includegraphics[width=\unitlength,page=1]{example_1.pdf}}%
    \put(0.10496378,0.01361326){\color[rgb]{0,0,0}\makebox(0,0)[lb]{\smash{$L$ $[+1]$}}}%
    \put(0.1275082,0.1759332){\color[rgb]{0,0,0}\makebox(0,0)[lb]{\smash{$L'$ $[+1]$}}}%
    \put(0.07340153,0.0917673){\color[rgb]{0,0,0}\makebox(0,0)[lb]{\smash{$H$}}}%
    \put(0.38150885,0.25108131){\color[rgb]{0,0,0}\makebox(0,0)[lb]{\smash{$L'$ $[0]$}}}%
    \put(0.38000589,0.01361326){\color[rgb]{0,0,0}\makebox(0,0)[lb]{\smash{$L$ $[0]$}}}%
    \put(0.34844368,0.0917673){\color[rgb]{0,0,0}\makebox(0,0)[lb]{\smash{$H$}}}%
    \put(0.67007764,0.09327026){\color[rgb]{0,0,0}\makebox(0,0)[lb]{\smash{$H$}}}%
    \put(0.41006513,0.13084432){\color[rgb]{0,0,0}\makebox(0,0)[lb]{\smash{$[-1]$ $E$}}}%
    \put(0.85794795,0.0752347){\color[rgb]{0,0,0}\makebox(0,0)[lb]{\smash{$E$ $[-2]$}}}%
    \put(0.69713096,0.25258428){\color[rgb]{0,0,0}\makebox(0,0)[lb]{\smash{$L'$ $[-1]$}}}%
    \put(0.7076517,0.01361326){\color[rgb]{0,0,0}\makebox(0,0)[lb]{\smash{$L$ $[0]$}}}%
    \put(0.23121262,0.06471398){\color[rgb]{0,0,0}\makebox(0,0)[lb]{\smash{$p$}}}%
    \put(0.48972217,0.20448948){\color[rgb]{0,0,0}\makebox(0,0)[lb]{\smash{$p'$}}}%
    \put(0.86997163,0.19246578){\color[rgb]{0,0,0}\makebox(0,0)[lb]{\smash{$E'$ $[-1]$}}}%
    \put(0,0){\includegraphics[width=\unitlength,page=2]{example_1.pdf}}%
  \end{picture}%
\endgroup%

%% file: example_2.pdf_tex
\begingroup%
  \makeatletter%
  \providecommand\color[2][]{%
    \errmessage{(Inkscape) Color is used for the text in Inkscape, but the package 'color.sty' is not loaded}%
    \renewcommand\color[2][]{}%
  }%
  \providecommand\transparent[1]{%
    \errmessage{(Inkscape) Transparency is used (non-zero) for the text in Inkscape, but the package 'transparent.sty' is not loaded}%
    \renewcommand\transparent[1]{}%
  }%
  \providecommand\rotatebox[2]{#2}%
  \ifx\svgwidth\undefined%
    \setlength{\unitlength}{1278.52113258bp}%
    \ifx\svgscale\undefined%
      \relax%
    \else%
      \setlength{\unitlength}{\unitlength * \real{\svgscale}}%
    \fi%
  \else%
    \setlength{\unitlength}{\svgwidth}%
  \fi%
  \global\let\svgwidth\undefined%
  \global\let\svgscale\undefined%
  \makeatother%
  \begin{picture}(1,0.57425266)%
    \put(0,0){\includegraphics[width=\unitlength,page=1]{example_2.pdf}}%
    \put(-0.00269383,0.00765369){\color[rgb]{0,0,0}\makebox(0,0)[lb]{\smash{$C_1=L'$}}}%
    \put(0.85413698,0.00785068){\color[rgb]{0,0,0}\makebox(0,0)[lb]{\smash{$C_2=E'$}}}%
    \put(0.75347594,0.4482198){\color[rgb]{0,0,0}\makebox(0,0)[lb]{\smash{$H$}}}%
    \put(0.47143108,0.36086107){\color[rgb]{0,0,0}\makebox(0,0)[lb]{\smash{$-K$}}}%
    \put(0.41035879,0.00659308){\color[rgb]{0,0,0}\makebox(0,0)[lb]{\smash{$L$}}}%
    \put(0.48419852,0.44834119){\color[rgb]{0,0,0}\makebox(0,0)[lb]{\smash{$K+r_2\Delta_2$}}}%
    \put(0.20172928,0.29280433){\color[rgb]{0,0,0}\makebox(0,0)[lb]{\smash{$K+r_1\Delta_1$}}}%
    \put(0,0){\includegraphics[width=\unitlength,page=2]{example_2.pdf}}%
    \put(0.31078387,0.40901005){\color[rgb]{0,0,0}\makebox(0,0)[lb]{\smash{$C_1^{\perp}$}}}%
    \put(0.75236539,0.19626427){\color[rgb]{0,0,0}\makebox(0,0)[lb]{\smash{$C_3^{\perp}$}}}%
    \put(0.62185751,0.24810986){\color[rgb]{0,0,0}\makebox(0,0)[lb]{\smash{$C_2^{\perp}$}}}%
    \put(0,0){\includegraphics[width=\unitlength,page=3]{example_2.pdf}}%
    \put(0.48062291,0.20162756){\color[rgb]{0,0,0}\makebox(0,0)[lb]{\smash{$\Al_4$}}}%
    \put(0.78990876,0.32498436){\color[rgb]{0,0,0}\makebox(0,0)[lb]{\smash{$\Al_3$}}}%
    \put(0,0){\includegraphics[width=\unitlength,page=4]{example_2.pdf}}%
    \put(0.61291861,0.12296525){\color[rgb]{0,0,0}\makebox(0,0)[lb]{\smash{$\Al_2$}}}%
    \put(0.34475168,0.06575628){\color[rgb]{0,0,0}\makebox(0,0)[lb]{\smash{$\Al_1$}}}%
    \put(0,0){\includegraphics[width=\unitlength,page=5]{example_2.pdf}}%
    \put(0.89538778,0.25704872){\color[rgb]{0,0,0}\makebox(0,0)[lb]{\smash{$\Cl$}}}%
  \end{picture}%
\endgroup%

%% file: example_3.pdf_tex
\begingroup%
  \makeatletter%
  \providecommand\color[2][]{%
    \errmessage{(Inkscape) Color is used for the text in Inkscape, but the package 'color.sty' is not loaded}%
    \renewcommand\color[2][]{}%
  }%
  \providecommand\transparent[1]{%
    \errmessage{(Inkscape) Transparency is used (non-zero) for the text in Inkscape, but the package 'transparent.sty' is not loaded}%
    \renewcommand\transparent[1]{}%
  }%
  \providecommand\rotatebox[2]{#2}%
  \ifx\svgwidth\undefined%
    \setlength{\unitlength}{1760.97223717bp}%
    \ifx\svgscale\undefined%
      \relax%
    \else%
      \setlength{\unitlength}{\unitlength * \real{\svgscale}}%
    \fi%
  \else%
    \setlength{\unitlength}{\svgwidth}%
  \fi%
  \global\let\svgwidth\undefined%
  \global\let\svgscale\undefined%
  \makeatother%
  \begin{picture}(1,0.24011098)%
    \put(0.08684858,0.00849692){\color[rgb]{0,0,0}\makebox(0,0)[lb]{\smash{$\partial^+\Cl$}}}%
    \put(0.3847853,0.00774948){\color[rgb]{0,0,0}\makebox(0,0)[lb]{\smash{$\Bl$}}}%
    \put(0,0){\includegraphics[width=\unitlength,page=1]{example_3.pdf}}%
    \put(0.73737129,0.00744578){\color[rgb]{0,0,0}\makebox(0,0)[lb]{\smash{$\Il$}}}%
    \put(0,0){\includegraphics[width=\unitlength,page=2]{example_3.pdf}}%
  \end{picture}%
\endgroup%

%% file: treeBassSerreGeneral.pdf_tex
\begingroup%
  \makeatletter%
  \providecommand\color[2][]{%
    \errmessage{(Inkscape) Color is used for the text in Inkscape, but the package 'color.sty' is not loaded}%
    \renewcommand\color[2][]{}%
  }%
  \providecommand\transparent[1]{%
    \errmessage{(Inkscape) Transparency is used (non-zero) for the text in Inkscape, but the package 'transparent.sty' is not loaded}%
    \renewcommand\transparent[1]{}%
  }%
  \providecommand\rotatebox[2]{#2}%
  \ifx\svgwidth\undefined%
    \setlength{\unitlength}{1166.4246855bp}%
    \ifx\svgscale\undefined%
      \relax%
    \else%
      \setlength{\unitlength}{\unitlength * \real{\svgscale}}%
    \fi%
  \else%
    \setlength{\unitlength}{\svgwidth}%
  \fi%
  \global\let\svgwidth\undefined%
  \global\let\svgscale\undefined%
  \makeatother%
  \begin{picture}(1,0.81355881)%
    \put(0,0){\includegraphics[width=\unitlength,page=1]{treeBassSerreGeneral.pdf}}%
    \put(-0.04234576,0.38677609){\color[rgb]{0,0,0}\makebox(0,0)[lb]{\smash{\SBd{$G_{i_1}$}}}}%
    \put(0.87804944,0.52081389){\color[rgb]{0,0,0}\makebox(0,0)[lb]{\smash{\SBd{$G_{i_k}$}}}}%
    \put(0.32170943,0.79001471){\color[rgb]{0,0,0}\makebox(0,0)[lb]{\smash{\SBd{$G_{i_3}$}}}}%
    \put(0.05379672,0.65587314){\color[rgb]{0,0,0}\makebox(0,0)[lb]{\smash{\SBd{$G_{i_2}$}}}}%
    \put(0.6603511,0.04908421){\color[rgb]{0,0,0}\makebox(0,0)[lb]{\smash{\SBd{$A$}}}}%
    \put(0.11013508,0.67220037){\color[rgb]{0,0,0}\makebox(0,0)[lb]{\smash{}}}%
    \put(0.71368884,0.06864661){\color[rgb]{0,0,0}\makebox(0,0)[lb]{\smash{}}}%
  \end{picture}%
\endgroup%

%% file: star1.pdf_tex
\begingroup%
  \makeatletter%
  \providecommand\color[2][]{%
    \errmessage{(Inkscape) Color is used for the text in Inkscape, but the package 'color.sty' is not loaded}%
    \renewcommand\color[2][]{}%
  }%
  \providecommand\transparent[1]{%
    \errmessage{(Inkscape) Transparency is used (non-zero) for the text in Inkscape, but the package 'transparent.sty' is not loaded}%
    \renewcommand\transparent[1]{}%
  }%
  \providecommand\rotatebox[2]{#2}%
  \ifx\svgwidth\undefined%
    \setlength{\unitlength}{2325.70306733bp}%
    \ifx\svgscale\undefined%
      \relax%
    \else%
      \setlength{\unitlength}{\unitlength * \real{\svgscale}}%
    \fi%
  \else%
    \setlength{\unitlength}{\svgwidth}%
  \fi%
  \global\let\svgwidth\undefined%
  \global\let\svgscale\undefined%
  \makeatother%
  \begin{picture}(1,0.72648673)%
    \put(0,0){\includegraphics[width=\unitlength,page=1]{star1.pdf}}%
    \put(0.23016675,0.33308736){\color[rgb]{0,0,0}\makebox(0,0)[lb]{\smash{\SB{$(\p^2,\id)$}}}}%
    \put(0.35558549,0.39433835){\color[rgb]{0,0,0}\makebox(0,0)[lb]{\smash{\SB{$(\p^2,b)$}}}}%
    \put(0.37803148,0.54277369){\color[rgb]{0,0,0}\makebox(0,0)[lb]{\smash{\SB{$(\p^2,\alpha_1b\alpha_1^{-1})$}}}}%
    \put(0.36059456,0.65880827){\color[rgb]{0,0,0}\makebox(0,0)[lb]{\smash{\SB{$(\p^2,\alpha_2b\alpha_2^{-1})$}}}}%
    \put(0.19091788,0.7160329){\color[rgb]{0,0,0}\makebox(0,0)[lb]{\smash{\SB{$(\p^2,b\alpha_1b\alpha_1^{-1})$}}}}%
    \put(0.0481257,0.67237733){\color[rgb]{0,0,0}\makebox(0,0)[lb]{\smash{\SB{$(\p^2,b\alpha_2b\alpha_2^{-1})$}}}}%
    \put(0.00773559,0.49309452){\color[rgb]{0,0,0}\makebox(0,0)[lb]{\smash{\SB{$(\p^2,b\alpha_2b\alpha_1^{-1})$}}}}%
    \put(0.02888375,0.35842436){\color[rgb]{0,0,0}\makebox(0,0)[lb]{\smash{\SB{$(\p^2,b\alpha_3b\alpha_2b\alpha_1^{-1})$}}}}%
    \put(0,0){\includegraphics[width=\unitlength,page=2]{star1.pdf}}%
    \put(0.36912288,0.16413983){\color[rgb]{0,0,0}\makebox(0,0)[lb]{\smash{\SB{$(\p^2,e_2)$}}}}%
    \put(0.34442589,0.28705398){\color[rgb]{0,0,0}\makebox(0,0)[lb]{\smash{\SB{$(\p^2,e_1)$}}}}%
    \put(0.36893268,0.06969535){\color[rgb]{0,0,0}\makebox(0,0)[lb]{\smash{\SB{$(\p^2,e_3)$}}}}%
    \put(0.20068266,0.00109513){\color[rgb]{0,0,0}\makebox(0,0)[lb]{\smash{\SB{$(\p^2,e_4)$}}}}%
    \put(0.04007313,0.04642511){\color[rgb]{0,0,0}\makebox(0,0)[lb]{\smash{\SB{$(\p^2,e_5)$}}}}%
    \put(-0.00000001,0.16648594){\color[rgb]{0,0,0}\makebox(0,0)[lb]{\smash{\SB{$(\p^2,e_6)$}}}}%
    \put(0.00363594,0.29259911){\color[rgb]{0,0,0}\makebox(0,0)[lb]{\smash{\SB{$(\p^2,e_7)$}}}}%
    \put(0.47380748,0.46598807){\color[rgb]{0,0,0}\makebox(0,0)[lb]{\smash{\SB{$(\p^2,e_1b)$}}}}%
    \put(0.63035893,0.50213001){\color[rgb]{0,0,0}\makebox(0,0)[lb]{\smash{\SB{$(\p^2,e_1\alpha_1b\alpha_1^{-1})$}}}}%
    \put(0,0){\includegraphics[width=\unitlength,page=3]{star1.pdf}}%
    \put(0.45391464,0.16857643){\color[rgb]{0,0,0}\makebox(0,0)[lb]{\smash{\SB{$(\p^2,e_1\alpha_3b\alpha_2b\alpha_1^{-1})$}}}}%
    \put(0.75137089,0.45292633){\color[rgb]{0,0,0}\makebox(0,0)[lb]{\smash{\SB{$(\p^2,e_1\alpha_2b\alpha_2^{-1})$}}}}%
    \put(0.78199642,0.33800144){\color[rgb]{0,0,0}\makebox(0,0)[lb]{\smash{\SB{$(\p^2,e_1b\alpha_1b\alpha_1^{-1})$}}}}%
    \put(0.7713124,0.20881005){\color[rgb]{0,0,0}\makebox(0,0)[lb]{\smash{\SB{$(\p^2,e_1b\alpha_2b\alpha_2^{-1})$}}}}%
    \put(0.28168479,-0.45490099){\color[rgb]{0,0,0}\makebox(0,0)[lb]{\smash{}}}%
    \put(0.60154061,0.11782124){\color[rgb]{0,0,0}\makebox(0,0)[lb]{\smash{\SB{$(\p^2,e_1b\alpha_2b\alpha_1)$}}}}%
    \put(0,0){\includegraphics[width=\unitlength,page=4]{star1.pdf}}%
    \put(0.44378632,0.60327369){\color[rgb]{0,0,0}\makebox(0,0)[lb]{\smash{\SB{$\Star(\Tl_b)$}}}}%
    \put(0.67059945,0.54845154){\color[rgb]{0,0,0}\makebox(0,0)[lb]{\smash{\SB{$\Star(e_1(\Tl_b))$}}}}%
    \put(0.46098542,0.06795168){\color[rgb]{0,0,0}\makebox(0,0)[lb]{\smash{\SB{$\Star(\Xl_e)$}}}}%
  \end{picture}%
\endgroup%